\documentclass[a4paper]{amsart}
\usepackage{verbatim}
\usepackage[dvips]{color}
\usepackage{amssymb}

\usepackage{lineno}

\usepackage[active]{srcltx}
\usepackage{latexsym}
\usepackage{color}
\usepackage{amsmath}
 \usepackage{float}
\usepackage{mathrsfs} %fonts veramente calligrafici con \mathscr{B}

\allowdisplaybreaks

\newtheorem{theorem}{Theorem}[section]

\newtheorem{lemma}{Lemma}[section]

\theoremstyle{definition}
\newtheorem{definition}{Definition}
\newtheorem{remark}{Remark}[section]

\floatstyle{ruled} \restylefloat{figure} \restylefloat{table}

\numberwithin{equation}{section}

\newcommand{\rn}[1]{\mathbf{R}^{#1}}

\newcommand{\beq}{\begin{equation}}
\newcommand{\bea}[1]{\begin{array}{#1} }
\newcommand{\eeq}{ \end{equation}}
\newcommand{\ea}{ \end{array}}
\newcommand{\ep}{\varepsilon}

\renewcommand{\k}{\kappa}

\newcommand{\ran}{\rangle}
\newcommand{\lan}{\langle}

\def \rn  {{\mathbb {R}}^{N}}
\def \R  {{\mathbb {R}}}

\def \p {{\partial}}

\def\mean#1{\mathchoice%
          {\mathop{\kern 0.2em\vrule width 0.6em height 0.69678ex depth -0.58065ex
                  \kern -0.8em \intop}\nolimits_{\kern -0.4em#1}}%
          {\mathop{\kern 0.1em\vrule width 0.5em height 0.69678ex depth -0.60387ex
                  \kern -0.6em \intop}\nolimits_{#1}}%
          {\mathop{\kern 0.1em\vrule width 0.5em height 0.69678ex
              depth -0.60387ex
                  \kern -0.6em \intop}\nolimits_{#1}}%
          {\mathop{\kern 0.1em\vrule width 0.5em height 0.69678ex depth -0.60387ex
                  \kern -0.6em \intop}\nolimits_{#1}}}

\def\vintslides_#1{\mathchoice%
          {\mathop{\kern 0.1em\vrule width 0.5em height 0.697ex depth -0.581ex
                  \kern -0.6em \intop}\nolimits_{\kern -0.4em#1}}%
          {\mathop{\kern 0.1em\vrule width 0.3em height 0.697ex depth -0.604ex
                  \kern -0.4em \intop}\nolimits_{#1}}%
          {\mathop{\kern 0.1em\vrule width 0.3em height 0.697ex depth -0.604ex
                  \kern -0.4em \intop}\nolimits_{#1}}%
          {\mathop{\kern 0.1em\vrule width 0.3em height 0.697ex depth -0.604ex
                  \kern -0.4em \intop}\nolimits_{#1}}}

\newcommand{\aveint}[2]{\mathchoice%
          {\mathop{\kern 0.2em\vrule width 0.6em height 0.69678ex depth -0.58065ex
                  \kern -0.8em \intop}\nolimits_{\kern -0.45em#1}^{#2}}%
          {\mathop{\kern 0.1em\vrule width 0.5em height 0.69678ex depth -0.60387ex
                  \kern -0.6em \intop}\nolimits_{#1}^{#2}}%
          {\mathop{\kern 0.1em\vrule width 0.5em height 0.69678ex depth -0.60387ex
                  \kern -0.6em \intop}\nolimits_{#1}^{#2}}%
          {\mathop{\kern 0.1em\vrule width 0.5em height 0.69678ex depth -0.60387ex
                  \kern -0.6em \intop}\nolimits_{#1}^{#2}}}

  \newcommand{\e}{\varepsilon}

\def\eqn#1$$#2$${\begin{equation}\label#1#2\end{equation}}
\def\charfn_#1{{\raise1.2pt\hbox{$\chi
_{\kern-1pt\lower3pt\hbox{{$\scriptstyle#1$}}}$}}}

\def\qq1{q_*}
\def\q2{q_{**}}

\def\ep{\varepsilon}

\def\rn{\mathbb R^N}

\newdimen\vintbar
\def\vint{-\kern-\vintbar\int}

\def\0{\boldsymbol 0}

\newtoks\by
\newtoks\paper
\newtoks\book
\newtoks\jour
\newtoks\yr
\newtoks\pages
\newtoks\vol
\newtoks\publ

\def\name[#1, #2]{#1 #2}
\def\ota{{\hbox{\bf ???}}}
\def\cLear{\by=\ota\paper=\ota\book=\ota\jour=\ota\yr=\ota
\pages=\ota\vol=\ota\publ=\ota}
\def\endpaper{\the\by, \textit{\the\paper},
{\the\jour} \textbf{\the\vol} (\the\yr), \the\pages.\cLear}
\def\endbook{\the\by, \textit{\the\book},
\the\publ, \the\yr.\cLear}
\def\endpap{\the\by, \textit{\the\paper}, \the\jour.\cLear}
\def\endproc{\the\by, \textit{\the\paper}, \the\book, \the\publ,
\the\yr, \the\pages.\cLear}

\begin{document}

%\modulolinenumbers[1]
%\linenumbers

\title[Variational inequalities for non-local operators]{Systems of variational inequalities for non-local operators
related to optimal switching problems: existence and uniqueness}

%\author[Lundstr\"{o}m]{Niklas L. P. Lundstr\"{o}m}
\address{Niklas L. P. Lundstr\"{o}m \\Department of Mathematics, Uppsala University\\
S-751 06 Uppsala, Sweden}
\email{niklas.lundstrom@math.umu.se}
%\author[Nystr\"{o}m]{Kaj Nystr\"{o}m}
\address{Kaj Nystr\"{o}m\\Department of Mathematics, Uppsala University\\
S-751 06 Uppsala, Sweden}
\email{kaj.nystrom@math.uu.se}
%\author[Olofsson]{Marcus Olofsson}
\address{Marcus Olofsson\\Department of Mathematics, Uppsala University\\
S-751 06 Uppsala, Sweden}
\email{marcus.olofsson@math.uu.se}
%}

\begin{abstract}
\medskip In this paper we study viscosity solutions to the system  \begin{eqnarray*}
&&\min\biggl\{-\mathcal{H}u_i(x,t)-\psi_i(x,t),u_i(x,t)-\max_{j\neq i}(-c_{i,j}(x,t)+u_j(x,t))\biggr\}=0,\notag\\
&&u_i(x,T)=g_i(x),\ i\in\{1,\dots,d\},
\end{eqnarray*}
where $(x,t)\in\mathbb R^{N}\times [0,T]$. Concerning $\mathcal{H}$ we assume that $\mathcal{H}=\mathcal{L}+\mathcal{I}$ where
$\mathcal{L}$ is a linear, possibly degenerate, parabolic operator of second order and $\mathcal{I}$ is a non-local integro-partial differential operator. A special case of this type of system of variational inequalities with terminal data occurs in the context of optimal switching problems when the
dynamics of the underlying state variables is described by an $N$-dimensional Levy process.  We establish a general comparison principle for viscosity sub- and supersolutions to the system under mild regularity, growth and
structural assumptions on the data, i.e., on the operator $\mathcal{H}$ and on continuous functions $\psi_i$, $c_{i,j}$, and
$g_i$.  Using the comparison principle we prove the existence of a unique viscosity solution $(u_1,\dots,u_d)$  to the system by Perron's method.
Our main contribution is that we establish existence and uniqueness of viscosity solutions,
in the setting of Levy processes and non-local operators,
with no sign assumption on the switching costs $\{c_{i,j}\}$ and allowing $c_{i,j}$ to depend on $x$ as well as $t$.

\noindent
2000  {\em Mathematics Subject Classification.}
\noindent

\medskip

\end{abstract}

\author{N. L. P. Lundstr\"{o}m, K. Nystr{\"o}m, M. Olofsson}\thanks{N.L.P. Lundstr\"{o}m and M. Olofsson were financed
by Jan Wallanders och Tom Hedelius Stiftelse samt Tore Browaldhs Stiftelse
through the project {\it Optimal switching problems and their applications in economics and finance}, P2010-0033:1.}

%\begin{center} \today
%\end{center}

\maketitle

\newpage

\setcounter{equation}{0} \setcounter{theorem}{0} \setcounter{definition}{0}
\section{Introduction and statement of main results}
\noindent In this paper we consider the problem
\begin{eqnarray}\label{eq4+}
&&\min\biggl\{-\mathcal{H}u_i(x,t)-\psi_i(x,t),u_i(x,t)-\max_{j\neq i}(-c_{i,j}(x,t)+u_j(x,t))\biggr\}=0,\notag\\
&&u_i(x,T)=g_i(x),\ i\in\{1,\dots,d\},
\end{eqnarray}
in $\mathbb R^N\times [0,T]$ ($(x,t)\in\mathbb R^N\times [0,T]$), $T>0$,  where $\psi_i$, $c_{i,j}$, and $g_i$ are continuous functions and
$\mathcal{H}$ is a non-local integro-partial differential operator. Concerning $\mathcal{H}$ we assume that $\mathcal{H}=\mathcal{L}+\mathcal{I}$,
\begin{equation}\label{opera}
    \mathcal{L}=\sum_{i,j=1}^N a_{i,j}(x,t)\partial_{x_i x_j}+\sum_{i=1}^N a_i(x,t)\partial_{x_i}+\partial_t,
\end{equation}
for continuous functions $a_{i,j}$ and $a_i$, and that $\mathcal{I}$ is a non-local integro-partial differential operator which for
 smooth functions $\phi$ is defined as
\begin{eqnarray}\label{operanloc}
   \mathcal{I}(x,t,\phi)&=&\int_{\mathbb R^l\setminus\{0\}} K(x,t,z,\phi,D\phi)\nu\left(dz\right),\notag\\
   K(x,t,z,\phi,p)&=&\phi\left(x+\eta\left(x,t,z\right),t\right) - \phi\left(x,t\right)\notag\\
    &&- \chi_{\{|z| \leq 1\}}(z) \sum_{k=1}^N\eta_k\left(x,t,z\right)p_k,
\end{eqnarray}
where $(x,t)\in\mathbb R^N\times\mathbb R$, $D\phi=(\partial_{x_1}\phi,...,\partial_{x_N}\phi)$,  $p\in\mathbb R^N$. Here $\nu$ is a positive Radon measure and, for each $i \in \{1,...,N\}$, $\eta_i$ is a function taking values in $\mathbb R^N$. $\chi_{\{|z| \leq 1\}}$ is the indicator function for the closed unit ball in  $\mathbb R^l$. Operators $\mathcal{H}=\mathcal{L}+\mathcal{I}$ occur, for instance, in the context of financial markets where the
dynamics of the state variables is described by an $N$-dimensional Levy process
$X=\left(X_s^{x,t}\right)$ solving the system of stochastic differential equations
\begin{eqnarray}\label{e-SDE}
dX_s^{x,t}&=&a(X_s^{x,t},s)  ds+\sigma(X_s^{x,t},s)dW_s\notag\\
&&+\int_{ z \in \R^l}\eta(X_{s^-}^{x,t},s,z) d\tilde N(ds,dz),
\end{eqnarray}
for $ t \leq s \leq T $, with initial condition $X_{s}^{x,t}=x$, $0\leq s \leq t$, $(x,t)\in\rn\times[0,T]$. Here $W=\{W_t\}$ denotes a standard $N$-dimensional Brownian motion and
\begin{equation}\label{eq:jumpmeas}
\tilde N(ds,dz) = \begin{cases} \hat{N}(ds,dz)&\mbox{if } |z| \geq 1 \\
\hat N(ds,dz) - \nu(dz) ds & \mbox{if } |z| <1, \end{cases}
\end{equation}
 where $\hat N$ is a Poisson random measure on  $[0,\infty) \times \R^l$ with intensity measure ${\nu(dz) ds}$.
$\mathcal{H}$ can, in the context of \eqref{e-SDE}, be seen as the infinitesimal generator associated to $X=\left(X_s^{x,t}\right)$ and we note that the diffusion part of the system then is described by the $N\times N$-dimensional matrix $\sigma$.
 The process
$X=\left(X_s^{x,t}\right)$ can, for instance, be the electricity price or other factors which determine the price.
In the Markovian setting when the randomness stems from the Levy process $X=\left(X_s^{x,t}\right)$ in \eqref{e-SDE}, the problem in \eqref{eq4+} is a system of variational inequalities with inter-connected obstacles related to multi-modes optimal switching problems. Our main results concern existence and uniqueness
of viscosity solutions to the system in \eqref{eq4+} under mild regularity, growth and
structural assumptions on the data, i.e., on the operator $\mathcal{H}$ and on continuous functions $\psi_i$, $c_{i,j}$, and
$g_i$.

In multi-modes optimal switching problems the system in \eqref{eq4+} occurs with $g=(g_1,\dots,g_d)\equiv(0,\dots,0)$.  To outline the setting
for this class of problems, consider a production facility which can run the  production in $d$, $d\geq 2$, production modes. Let $X=\left(X_s^{x,t}\right)$ denote the vector of stochastic processes in \eqref{e-SDE} which, as discussed above,  represents the market price of the underlying commodities and other finance assets that influence the production. Let the payoff rate in production mode $i$, at time $t$, be $\psi_i(X_t,t)$ and let $c_{i,j}(X_t,t)$ denote the switching cost for switching from mode $i$ to mode $j$ at time $t$. A management strategy is a combination of a non-decreasing sequence of stopping times $\{\tau_k\}_{k\geq 0}$, where, at time $\tau_k$, the manager decides to switch production from its current mode to another one, and a sequence of indicators $\{\xi_k\}_{k\geq 0}$, taking values in $\{1,\dots,d\}$, indicating the mode to which the production is switched. At $\tau_k$  the production is switched from mode $\xi_{k-1}$ (current mode) to $\xi_k$. A strategy $(\{\tau_k\}_{k\geq 0},\{\xi_k\}_{k\geq 0})$ can be represented by the simple function
$$
\mu_s=\sum_{i\geq 1} \xi_i \chi_{(\tau_i, \tau_{i+1} ]}(s) + \xi_0 \chi_{[\tau_0, \tau_1 ]}(s).$$
When the production is run using a strategy $\mu$, defined by $(\{\tau_k\}_{k\geq 0},\{\xi_k\}_{k\geq 0})$, over a finite horizon $[0,T]$, the total expected
profit up to time $T$ is
\begin{eqnarray}\label{eq1}
J(\mu)=E\biggl [\biggl(\int\limits_0^T\psi_{\mu_s}(X_s,s)ds-\sum_{k\geq 1}c_{\xi_{{k-1}},\xi_{k}}(X_{\tau_{k}},\tau_{k})\biggr )\biggr ].
\end{eqnarray}
The optimal switching problem consists of finding an optimal management strategy
$\mu^\ast$, defined by $(\{\tau_k^\ast\}_{k\geq 0},\{\xi_k^\ast\}_{k\geq 0})$, such that
\begin{eqnarray}\label{eq2}
J(\mu^\ast)=\sup_{\mu}J(\mu).
\end{eqnarray}
Let $(Y_t^1,\dots,Y_t^d)$ be the value function associated with the optimal switching problem, on the time interval $[t,T]$,
where $Y_t^i$ stands for the optimal expected profit if, at time $t$,
the production is in mode $i$. Under sufficient assumptions, see for example \cite{BJK10}, it can then be proved that $(Y_t^1,\dots,Y_t^d)=(u_1(X_t,t),\dots,u_d(X_t,t))$,
where the vector of deterministic functions $(u_1(x,t),\dots,u_d(x,t))$ satisfies \eqref{eq4+} with $g=(g_1,\dots,g_d)\equiv(0,\dots,0)$.

\subsection{Assumptions} We here outline the assumptions we impose on  $\mathcal{H}$ and on the functions $\psi_i$, $c_{i,j}$, and
$g_i$. Firstly, focusing on $\mathcal{H}=\mathcal{L}+\mathcal{I}$ we impose additional structural assumptions on the matrix
$\{a_{i,j}\}_{i,j=1}^N$. %assumptions which seem difficult to avoid when constructing solutions along the lines outlined below and using the formalism of viscosity solutions.
In particular, we assume that
\begin{eqnarray}\label{assump1el}
a_{i,j}(x,t)=(\sigma(x,t)\sigma^\ast(x,t))_{i,j},\ i,j\in\{1,\dots,N\},
\end{eqnarray}
where $\sigma=\sigma(x,t)$ is an $N\times N$ matrix and $\sigma^\ast$ is the transpose of $\sigma$. Concerning regularity and growth conditions on
$\{a_{i,j}\}_{i,j=1}^N$ and $\{a_{i}\}_{i=1}^N$, we assume that
\begin{eqnarray}\label{assump1elel}
(i)&&|a_i(x,t)-a_i(y,s)|+|\sigma_{i,j}(x,t)-\sigma_{i,j}(y,s)|\leq A|x-y|,\notag\\
(ii)&& |a_i(x,t)|+|\sigma_{i,j}(x,t)|\leq A(1+|x|),
\end{eqnarray}
for some $A$, $1\leq A<\infty$, for all $ i,j\in\{1,\dots,N\}$, and whenever $(x,t),(y,s)\in\mathbb R^N\times [0,T]$.
Here, $|x|$ is the standard Euclidean norm of $x\in\mathbb R^N$. Note that $(i)$ implies $(ii)$ in \eqref{assump1elel}.
Note also that $\{a_{i,j}\}_{i,j=1}^N$ is only assumed non-negative definite and hence large sets of entries in the
matrices $\{a_{i,j}\}_{i,j=1}^N$,  $\{\sigma_{i,j}\}_{i,j=1}^N$, may be zero at points resulting in degeneracies.

 The above restrictions define the local part, $\mathcal{L}$, of $\mathcal{H}$. Focusing on the non-local part, $\mathcal{I}$, of $\mathcal{H}$, we assume that $\nu$ is a positive Radon measure defined on $\mathbb R^l\setminus\{0\}$ such that
\begin{eqnarray}\label{assump1elelnloc}
\int_{0<|z|\leq 1} |z|^2 \nu\left(dz\right) + \int_{1<|z|} e^{\Lambda|z|} \nu\left(dz\right) \leq \tilde A
\end{eqnarray}
for some constant $\tilde A$, $1\leq \tilde A<\infty$, and for some $\Lambda>0$. In addition we assume that $\eta_k$ is, for $k\in\{1,...,N\}$,
continuous in $x$ and $t$, Borel measurable in $z$,  and that
\begin{eqnarray}\label{assump1elelnloc+}
|\eta_k\left(x,t , z\right)| &\leq&\tilde B\min\{|z|,1\},\notag \\
|\eta_k\left(x, t, z\right) - \eta_k\left(y, t, z\right)| &\leq &\tilde B \min\{|z|,1\}|x - y|,
\end{eqnarray}
for some constant $\tilde B$, $1\leq \tilde B<\infty$, and for all $x,y\in\mathbb R^N$, $t\in\mathbb R$, $z\in\mathbb R^l$. This completes our definition of $\mathcal{H}$. Secondly, concerning regularity and growth conditions on $\psi_i$, $c_{i,j}$ and $g_i$ we assume that
\begin{eqnarray}\label{assump2}
&&\mbox{$\psi_i$, $c_{i,j}$ and $g_i$ are continuous functions,}\notag\\
&&\mbox{$|\psi_i(x,t)|+|c_{i,j}(x,t)|+|g_{i}(x)|\leq B(1+|x|^\gamma)$ for some }\notag\\
&&\mbox{$B,\gamma\in [1,\infty)$, whenever $(x,t)\in \mathbb R^N\times [0,T]$}.
\end{eqnarray}
Thirdly, the structural assumptions on the functions $\{c_{i,j}\}$ that we impose to establish our general
comparison principle for the system in \eqref{eq4+}, see Theorem \ref{Thm1-}, are the following.
 \begin{eqnarray}\label{assump3}
 (i)&&c_{i,i}(x,t)=0\mbox{ for each $i\in\{1,\dots.,d\}$},\notag\\
 (ii)&&\mbox{For any sequence $i_1$,\dots, $i_k$, $i_j\in\{1,\dots,d\}$ for each $j\in\{1,\dots,k\}$,}\notag\\
 &&\mbox{we have $c_{i_1,i_2}(x,t)+c_{i_2,i_3}(x,t)+\dots+c_{i_{k-1},i_k}(x,t)+c_{i_k,i_1}(x,t)>0$},\notag\\
 &&\mbox{for all $(x,t)\in \mathbb R^N\times [0,T]$.}
\end{eqnarray}
 Finally, concerning the interplay between the terminal data $\{g_i\}$ and the switching costs $\{c_{i,j}\}$, at $t=T$, we assume that
 \begin{eqnarray}\label{assump2el}
 g_i(x)\geq\max_{j\neq i}(-c_{i,j}(x,T)+g_j(x))
 \end{eqnarray}
 for all $ i\in\{1,\dots,d\}$, and for all $x\in\mathbb R^N$.  Note that in the special case of the optimal switching problem discussed above, i.e., $g=(g_1,\dots,g_d)\equiv(0,\dots,0)$,
 then \eqref{assump2el} implies that $c_{i,j}(x,T)\geq 0$ for all $ i,j\in\{1,\dots,d\}$, and for all $x\in\mathbb R^N$.  Focusing on the structural assumptions on the functions $\{c_{i,j}\}$, we emphasize that in our final existence theorem, see Theorem \ref{Thm1+}, we assume regularity
of $c_{i,j}(x,t)$ beyond continuity, see \eqref{assump2+}, and that
\begin{eqnarray}\label{assump3+}
c_{i_1,i_2}(x,t)+c_{i_2,i_3}(x,t)\geq c_{i_1,i_3}(x,t)
\end{eqnarray}
for any sequence of indices $i_1$, $i_2$, $i_3$, $i_l\in\{1,\dots,d\}$ for each $l\in\{1,2,3\}$, and for all $(x,t)\in \mathbb R^N\times [0,T]$. Note that
\eqref{assump3+} is an additional structural restriction compared to \eqref{assump3}.
%In fact, \eqref{assump3+} together with \eqref{assump3} $(ii)$ implies \eqref{assump3} $(i)$.%

\subsection{Statement of main results} We here formulate our main results. For the definition of $\mbox{LSC}_{p}(\mathbb R^N\times [0,T])$, $\mbox{USC}_{p}(\mathbb R^N\times [0,T])$ and $\mbox{C}_{p}(\mathbb R^N\times [0,T])$ as well as for the definition of viscosity sub- and supersolutions we refer to the bulk of the paper.
We first prove the following comparison principle.

\begin{theorem}\label{Thm1-}
Let $\mathcal{H}=\mathcal{L}+\mathcal{I}$ with $\mathcal{L}$, $\mathcal{I}$, as in \eqref{opera}, \eqref{operanloc}, respectively.
Assume \eqref{assump1el}, \eqref{assump1elel}, \eqref{assump1elelnloc}, \eqref{assump1elelnloc+}, \eqref{assump2}, \eqref{assump3} and \eqref{assump2el}.
Assume that $( u_1^+,\dots,u_d^+)\in \mbox{LSC}_{p}(\mathbb R^N\times [0,T])$ and $(u_1^-,\dots,u_d^-)\in \mbox{USC}_{p}(\mathbb R^N\times [0,T])$ are viscosity super- and subsolutions, respectively,  to the problem in \eqref{eq4+}.
Then $u_i^-\leq u_i^+$ in $\mathbb R^N\times (0,T]$ for all $i\in\{1,\dots,d\}$.
\end{theorem}

Before stating our existence theorems, we make the following definition.

\begin{definition} \label{barr1}
A barrier from above for the system in \eqref{eq4+}, component $i \in \{1,\dots, d \}$ and the point $y\in \R^N$, $u^{+,i,y}$, is a family of continuous supersolutions, $\{u^{+,i,y,\e}\}_{\e>0}$, $u^{+,i,y,\e}\in \mbox{C}_{p}(\mathbb R^N\times [0,T])$, to system \eqref{eq4+} such that $$\lim_{\e\to 0} u_i^{+,i,y,\e}(y,T) = g_i(y).$$ A barrier from below for the system in \eqref{eq4+}, component $i \in \{1,\dots, d \}$ and the point $y\in \R^N$, $u^{-,i,y}$, is a family of continuous subsolutions, $\{u^{-,i,y,\e}\}_{\e>0}$, $u^{-,i,y,\e}\in \mbox{C}_{p}(\mathbb R^N\times [0,T])$, to system \eqref{eq4+} such that $$\lim_{\e\to 0} u_i^{-,i,y,\e}(y,T) = g_i(y).$$
\end{definition}

\noindent
To stress generality, we first prove the following theorem.

\begin{theorem}\label{Thm1aa} Let $\mathcal{H}=\mathcal{L}+\mathcal{I}$ with $\mathcal{L}$, $\mathcal{I}$, as in \eqref{opera}, \eqref{operanloc}, respectively.
Assume \eqref{assump1el}, \eqref{assump1elel}, \eqref{assump1elelnloc}, \eqref{assump1elelnloc+}, \eqref{assump2},\eqref{assump3} and \eqref{assump2el}. In addition, assume that
\begin{enumerate}
  \item there exists, for
each $i\in\{1,\dots,d\}$ and $y\in \R^N$, a barrier from above  $u^+=u^{+,i,y}$ to the system in \eqref{eq4+} in the sense of Definition \ref{barr1},
  \item there exists, for
each $i\in\{1,\dots,d\}$ and $y\in \R^N$, a barrier from below $u^-=u^{-,i,y}$ to the system in \eqref{eq4+} in the sense of Definition \ref{barr1}.
\end{enumerate}
Then there exists a viscosity solution $(u_1,\dots,u_d)\in \mbox{C}_{p}(\mathbb R^N\times [0,T])$ to the problem in
\eqref{eq4+}, $ u_i^-\leq u_i\leq u_i^+$ on $\mathbb R^N\times [0,T]$, for $i\in\{1,\dots,d\}$, and this solution is unique in the class $ \mbox{C}_{p}(\mathbb R^N\times [0,T])$.
\end{theorem}

To establish existence of a viscosity solution $(u_1,\dots,u_d)$ to the problem in
\eqref{eq4+}, it hence remains to construct barriers in the sense of Definition \ref{barr1}.
To do so we impose additional assumptions on the switching costs $c_{i,j}$, see Theorem \ref{Thm1+} below.
In particular, we prove that for fixed $i \in \{1,\dots, d\}$, $y \in \R^N$ and all $j\in\{1,\dots,d\}$,
\begin{eqnarray*}
u_j^{+,i,y,\e}(x,t) &=& g(y)+\frac{ K}{\e^2}(T-t)\notag\\
&& + L(e^{\lambda(T-t)}+1)(|x-y|^2 + \e)^{\frac 12} + c_{i,j}(x,t),
\end{eqnarray*}
is a barrier from above if $K$ and $\lambda$ are large enough.
Here $L$ is the Lipschitz-constant of $g(x)$.
Since $c_{i,i} = 0$ by assumption, $u_j^{+,i,y,\e}$ attains the terminal data $g$ as $\ep \to 0$ and hence,
from Theorem \ref{Thm1aa}, we now deduce existence of a viscosity solution to the problem in \eqref{eq4+}.
In particular, we prove the following theorem.

\begin{theorem}\label{Thm1+} Let $\mathcal{H}=\mathcal{L}+\mathcal{I}$ with $\mathcal{L}$, $\mathcal{I}$, as in \eqref{opera}, \eqref{operanloc}, respectively.
Assume \eqref{assump1el}, \eqref{assump1elel}, \eqref{assump1elelnloc}, \eqref{assump1elelnloc+}, \eqref{assump2}, \eqref{assump3} and \eqref{assump2el}. Assume also that $g_i=g$, for all $i \in \{1,\dots, d\}$, for some Lipschitz continuous function $g$, that $ c_{i,j} \in C^{1,2}(\mathbb R^N\times [0,T))$, and that \eqref{assump3+} holds
for any sequence of indices $i_1$, $i_2$, $i_3$, $i_j\in\{1,\dots.,d\}$ for each $j\in\{1,2,3\}$, and for all $(x,t)\in \mathbb R^N\times [0,T]$. Furthermore, assume that
\begin{eqnarray}\label{assump2+}
 &&\mbox{$ \partial _{x_k} c_{i,j} \in L^\infty(\mathbb R^N\times [0,T])$} \notag \\
&&\mbox{$|\psi_i(x,t)|, |g(x,t)|, |\partial_t c_{i,j}(x,t)|,  |\partial_{x_k x_l}c_{i,j}(x,t)|(1+|x|^2)  \leq \tilde A( 1 + |x|)$}
 \end{eqnarray}
 for all $(x,t)\in \mathbb R^N\times [0,T]$ and for some $\tilde A$, $1\leq\tilde A<\infty$. Then there exists a unique viscosity solution $(u_1,\dots,u_d)$ to the problem in
\eqref{eq4+},  unique in the sense defined in Theorem \ref{Thm1aa}.
\end{theorem}

\begin{remark} Through $\mathcal{H}=\mathcal{L}+\mathcal{I}$ we see that $\mathcal{H}$ is written as the sum of the local operator $\mathcal{L}$ and the non-local operator
 $\mathcal{I}$. Note that by the assumptions in Theorem \ref{Thm1-}-Theorem \ref{Thm1+} the matrix $\{a_{i,j}(x,t)\}$ is only assumed to be non-negative definite and as such it can vanish at points. Similarly the jump vector $\eta=\eta(x,t,z)$ is allowed to vanish. A consequence of this is that there are no regularization effects in the problem coming either from  $\mathcal{L}$ or  $\mathcal{I}$. Therefore the system in \eqref{eq4+} can not be expected to have classical solutions and hence an appropriate notion of viscosity solutions is needed.
\end{remark}

\begin{remark} $\mathcal{H}$ can be seen as the infinitesimal generator associated with a Levy process described by \eqref{e-SDE}. In this context, the assumptions stated in \eqref{assump1elelnloc} exclude some Levy processes as a  Levy measure $\nu$ in general only satisfies
$$\int \min\{|z|^2,1\}\nu(dz) < \infty.$$ However, for many, if not most, applications the class of Levy processes considered in this paper is sufficiently rich since, e.g., any Levy process with compactly supported Levy measure satisfies \eqref{assump1elelnloc}.
\end{remark}

\begin{remark} In \cite{BJK10} the authors basically prove Theorem \ref{Thm1-}-Theorem \ref{Thm1+} for a combined optimal switching and control problem essentially assuming
$\mathcal{H}=\mathcal{L}+\mathcal{I}$ with $\mathcal{L}$, $\mathcal{I}$, as in \eqref{opera}, \eqref{operanloc}, respectively, and \eqref{assump1el}, \eqref{assump1elel}, \eqref{assump1elelnloc}, \eqref{assump1elelnloc+}, and \eqref{assump2} with $\gamma=2$. Our contribution is that we prove Theorem \ref{Thm1-}-Theorem \ref{Thm1+} allowing for much more general switching costs compared to \cite{BJK10}.
\end{remark}

\begin{remark} Our proofs of Theorem \ref{Thm1-}-Theorem \ref{Thm1+} are influenced by the corresponding arguments in \cite{BJK10}, \cite{AF12},
and by the arguments in \cite{LNO12} where versions of \ref{Thm1-}-Theorem \ref{Thm1+} are proved in the case when
$\mathcal{H}\equiv\mathcal{L}$, i.e., in the case of local operators. In the latter paper the problem of regularity of viscosity solutions was also treated in the context of operators of Kolmogorov type.
\end{remark}

%
%\begin{remark} In the case $\mathcal{H}\equiv\mathcal{L}$, i.e., in the case of local operators, the assumption that \eqref{assump2} should hold with $\gamma=2$ can be relaxed and general polynomial growth can be allowed, see \cite{LNO12}. The reason that we in the case of non-local operators impose a more restrictive growth condition is that we base our proof on \cite{JK05}, in particular Theorem 2.2 in \cite{JK05}, see also \cite{JK06}. Currently it is not clear to us how to generalize our results beyond this assumption. However, for applied purposes, this assumption is not a vital restriction.
%\end{remark}

\begin{remark} Theorem \ref{Thm1-}-Theorem \ref{Thm1+} remain true for the more general non-linear problems 
\begin{eqnarray}\label{eq4+bla}
&&\min\biggl\{- \sup_{\alpha\in A_i} \inf_{\beta\in B_i}\big[\mathcal{H}^i_{\alpha, \beta}u_i(x,t)-\psi_i^{\alpha, \beta}(x,t)\big],\\
&&u_i(x,t)-\max_{j\neq i}(-c_{i,j}(x,t)+u_j(x,t))\biggr\}=0,\notag\\
&&u_i(x,T)=g_i(x),\ i\in\{1,\dots,d\},
\end{eqnarray}
where $A_i, B_i$ are compact metric spaces and $\mathcal{H}^i_{\alpha, \beta}=\mathcal{L}_{\alpha, \beta}^i+\mathcal{I}_{\alpha, \beta}^i$, for each $i\in\{1,\dots,d\}$, is as in \eqref{opera}, \eqref{operanloc}, for some
$\sigma_{k,l,\alpha, \beta}^i$, $a_{k,l,\alpha, \beta}^i$, $a_{k,\alpha, \beta}^i$, $\nu$, $\eta^i_{\alpha, \beta}$ satisfying \eqref{assump1el}, \eqref{assump1elel}, \eqref{assump1elelnloc}, and \eqref{assump1elelnloc+} uniformly in $\alpha, \beta$. 
Equations of this type are considered in \cite{BJK10} and arise for example as the Bellman-Isaacs equations for zero-sum stochastic games.
\end{remark}

\subsection{Our contribution in relation to the current literature} In this paper we consider
optimal switching problems for operators $\mathcal{H}=\mathcal{L}+\mathcal{I}$, where $\mathcal{L}$ is a  local operator and  $\mathcal{I}$ is non-local operator, allowing for general switching costs. As a general comment we note that there is currently a substantial literature devoted
 to the the case of local operators, i.e., $\mathcal{H}\equiv\mathcal{L}$, and that considerably less is known in the case when non-local effects are allowed.

 In the local setting there is a well-established connection between the theory of reflected backward stochastic differential equations driven by Brownian motion and multi-modes optimal switching problems, see  \cite{AF12}, \cite{DH09}, \cite{DHP10}, \cite{HT07} and \cite{HZ10}. In general the research on optimal switching problems focuses on two approaches: the probabilistic or stochastic approach and the deterministic approach.  We note that the stochastic approach heavily uses probabilistic tools such as the Snell envelope and backward stochastic differential equations. For more on the stochastic approach we refer to the references above concerning reflected backward stochastic differential equations. The latter, deterministic approach focuses more on variational inequalities and partial differential equations and we refer, in the local setting, to  \cite{AH09} and \cite{LNO12}. In general, in the local setting the two approaches are used in combination due to the connection between reflected backward stochastic differential equations and systems of variational inequalities. The results in the literature concerning the optimal switching problem, i.e., system \eqref{eq4+} with $g=(g_1,\dots,g_d)\equiv(0,\dots,0)$, are derived under different assumptions on the switching costs $\{c_{i,j}\}$. We note that the switching costs in the references listed above are essentially always assumed to be non-negative, i.e., $c_{i,j}(x,t)\geq \alpha$ for all $(x,t)\in\mathbb R^N\times [0,T]$,
$i,j\in\{1,\dots,d\}$, $i\neq j$, and for some $\alpha \geq 0$. Often, even $\alpha>0$ is assumed, see \cite{AH09} and \cite{DH09} for instance. Furthermore, often additional restrictions are imposed on $\{c_{i,j}\}$, like $c_{i,j}$ is independent of $x$, see \cite{DHP10} for instance, or that $c_{i,j}$ is even constant, see \cite{DH09}, \cite{BJK10}.
The only papers we are aware of where the switching costs are allowed to change sign are \cite{AF12} and \cite{LNO12}. However, in \cite{AF12} there are two additional conditions concerning the non-negativity of the switching costs: $c_{i,j}(x, T ) = 0$ and the number of negative switching costs is limited in a certain sense, see condition $(v)$ in \cite{AF12}. In \cite{LNO12} the assumptions imposed on the switching costs are in line with the assumptions imposed in this paper.

In the non-local setting the connections between a theory of reflected backward stochastic differential equation driven by Levy processes and multi-modes optimal switching problem seems to be considerably  less developed. In general the theory for non-local operators is currently an active area of research. We here refer to \cite{A07}, \cite{BI08}, \cite{BJK10} and \cite{P97} for an account of this and for references, but to our knowledge \cite{BJK10} is the only paper that makes contribution to multi-mode optimal switching problems involving non-local operators, i.e., involving non-local effects through Levy processes that are allowed to jump. Building on \cite{BJK10} our main contribution is that while in \cite{BJK10} the switching costs are assumed to be non-negative
constants, our baseline assumption on the functions $\{c_{i,j}\}$ is that we make no sign assumption on the switching costs $\{c_{i,j}\}$ and that $c_{i,j}$ is allowed to depend on $x$ as well as $t$. Naturally one may ask whether the allowance for possibly negative switching costs is of importance beyond the mathematical issues. In fact, it is very natural to allow for negative switching costs as these allow one to model the situation when, for example, a government through environmental policies provide subsidies, grants or other financial support to energy production facilities in case they switch to more `green energy' production or to more clean methods for production. In this case it is not a cost for the facility to switch, it is a gain. However, the final decision to switch or not to switch is naturally also influenced by the running pay-offs $\{\psi_i\}$.

\subsection{Outline of the paper}
The rest of this paper is organized as follows. Section \ref{sec:preliminaries} is of preliminary nature and we here introduce some notation and state the definition of viscosity subsolutions, supersolutions and solutions. Section \ref{sec:compprinciple} is devoted to the proof of Theorem \ref{Thm1-} and in  section \ref{sec:exuni} we prove  Theorem \ref{Thm1aa} and Theorem \ref{Thm1+}.

\setcounter{equation}{0} \setcounter{theorem}{0}
\section{Preliminaries}\label{sec:preliminaries}
\noindent
In this section we introduce some notation used throughout the paper and we define viscosity sub- and supersolutions to the problem in \eqref{eq4+}.
\subsection{Notation} We denote by $\mbox{LSC}(\mathbb R^N\times [0,T])$ the set of lower semi-continuous functions, i.e., all functions $f: \mathbb R^N\times [0,T] \to \R$ such that for all points $(x_0,t_0)$ and for any sequence $\{(x_n, t_n)\}_n$,  $\lim _{n \to \infty} (x_n,t_n) \to (x_0,t_0)$ in $\mathbb R^N\times [0,T]$, we have
\begin{equation*}
\liminf _{n \to \infty} f(x_n,t_n) \geq f(x_0,t_0).
\end{equation*}
Likewise, we denote by
$\mbox{USC}(\mathbb R^N\times [0,T])$ the set of upper semi-continuous functions, i.e., all functions $f: \mathbb R^N\times [0,T]\to \R$ such that for all points $(x_0,t_0
)$ and for any sequence $\{(x_n,t_n)\}_n$,  $\lim _{n \to \infty} (x_n,t_n) \to (x_0,t_0)$ in $\mathbb R^N\times [0,T]$, we have
\begin{equation*}
\limsup _{n \to \infty}f(x_n,t_n) \leq f(x_0,t_0).
\end{equation*}
Note that a function $f$ is upper semi-continuous if and only if $-f$ is lower semi-continuous. Also, a real function $g$ is continuous if and only if it is both upper and lower semi-continuous. Given $\gamma\in [1,\infty)$, the function space $\mbox{LSC}_{p,\gamma}(\mathbb R^N\times [0,T])$ is defined to consist
of functions $h\in \mbox{LSC}(\mathbb R^N\times [0,T])$ which satisfy the growth condition
\begin{eqnarray} \label{eq:growth}
|h(x,t)| \leq  c(1+|x|^\gamma)
\end{eqnarray}
for some $c\in [1,\infty)$, whenever $(x,t)\in \mathbb R^N\times [0,T]$. $\mbox{USC}_{p,\gamma}(\mathbb R^N\times [0,T])$ is defined by analogy. Furthermore, $\mbox{C}_{p,\gamma}(\mathbb R^N\times [0,T])=\mbox{USC}_{p,\gamma}(\mathbb R^N\times [0,T])\cap \mbox{LSC}_{p,\gamma}(\mathbb R^N\times [0,T])$. $\mbox{C}_{p}(\mathbb R^N\times [0,T])$, $\mbox{USC}_{p}(\mathbb R^N\times [0,T])$, $\mbox{LSC}_{p}(\mathbb R^N\times [0,T])$ are the unions, respectively, of the sets $\mbox{C}_{p,\gamma}(\mathbb R^N\times [0,T])$, $\mbox{USC}_{p,\gamma}(\mathbb R^N\times [0,T])$, $\mbox{LSC}_{p,\gamma}(\mathbb R^N\times [0,T])$, with respect to $\gamma\geq 1$. By $C^{1,2}(\mathbb R^N\times [0,T))$ we denote the set of functions which are twice continuously differentiable in the spatial variables and once continuously differentiable in the time variable,  on $\mathbb R^N\times [0,T)$. We will by $c$ denote a generic constant, $1\leq c < \infty$, that may change value from line to line.

\subsection{Viscosity solutions} We here define the notion of viscosity solutions to the problem in \eqref{eq4+}.
Let $\mathcal{H}=\mathcal{L}+\mathcal{I}$ with $\mathcal{L}$, $\mathcal{I}$, as in \eqref{opera}, \eqref{operanloc}, respectively. Given
$(x,t)\in \mathbb R^N\times [0,T]$, $\phi\in C^{1,2}(\mathbb R^N\times [0,T))$, $p\in\mathbb R^N$, $\kappa\in (0,1)$, we let
\begin{eqnarray}\label{operanlocaa}
   \mathcal{I}_\kappa(x,t,\phi,p)&=&\int_{\{z: |z|<\kappa\}} K(x,t,z,\phi,p)\nu\left(dz\right),\notag\\
   \mathcal{I}^\kappa(x,t,\phi,p)&=&\int_{\{z: |z|\geq\kappa\}} K(x,t,z,\phi,p)\nu\left(dz\right),
\end{eqnarray}
where
\begin{eqnarray}\label{operanlocll}
   K(x,t,z,\phi,p)&=&\phi\left(x+\eta\left(x,t,z\right),t\right) - \phi\left(x,t\right)\notag\\
    &&- \chi_{\{|z| \leq 1\}}(z) \sum_{k=1}^N\eta_k\left(x,t,z\right)p_k.
\end{eqnarray} Note that
$\mathcal{I}=\mathcal{I}_\kappa+\mathcal{I}^\kappa$ and that $\mathcal{I}_\kappa$ and $\mathcal{I}^\kappa$ give, respectively,
the contribution to $\mathcal{I}$ from the `small' and `large' jumps. Using this notation we let $\mathcal{H}^\kappa(\phi,u)=\mathcal{H}^\kappa(x,t,\phi,u)$ where
\begin{eqnarray}\label{operanlocaa+}
\mathcal{H}^\kappa(x,t,\phi,u)&:=&\mathcal{L}\phi(x,t)+\mathcal{I}_\kappa(x,t,\phi,D\phi)+\mathcal{I}^\kappa(x,t,u,D\phi)
\end{eqnarray}
whenever $(x,t)\in \mathbb R^N\times [0,T]$, $\phi\in C^{1,2}(\mathbb R^N\times [0,T))$ and
$u\in\mbox{LSC}_p(\mathbb R^N\times [0,T])\cup \mbox{USC}_p(\mathbb R^N\times [0,T])$.

\begin{definition}\label{def:viscosity} $(i)$ A vector $(u_1^+,\dots,u_d^+)$, $u_i\in\mbox{LSC}_p(\mathbb R^N\times [0,T])$ for $i\in\{1,\dots,d\}$, is a viscosity supersolution to the problem in \eqref{eq4+} if $u_i^+(x,T)\geq g_i(x)$  whenever $x\in\mathbb R^N$, $i\in\{1,\dots,d\}$, and if the following holds. If $(x_0,t_0)\in\mathbb R^N\times (0,T)$ and if, for some $i\in\{1,\dots,d\}$, we have $\phi_i\in C^{1,2}(\mathbb R^N\times [0,T))$ such that
\begin{eqnarray*}
(i)&&\phi_i(x_0, t_0) = u_i^+(x_0,t_0),\notag\\
(ii)&&(x_0,t_0)\mbox{ is a global maximum of $\phi_i-u_i^+$},
\end{eqnarray*}
then
\begin{eqnarray*}
&&\min\biggl\{-\mathcal{H}^\kappa(\phi_i,u_i^+)-\psi_i, u_i^+-\max_{j\neq i}(-c_{i,j}+u_j^+)\biggr\} \geq 0\mbox{ for all }\kappa\in (0,1).
\end{eqnarray*}
$(ii)$ A vector $(u_1^-,\dots,u_d^-)$, $u_i^-\in\mbox{USC}_p(\mathbb R^N\times [0,T])$ for $i\in\{1,\dots,d\}$, is a viscosity subsolution to the problem in \eqref{eq4+} if $u_i^-(x,T)\leq g_i(x)$  whenever $x\in\mathbb R^N$, $i\in\{1,\dots,d\}$, and if the following holds. If $(x_0,t_0)\in\mathbb R^N\times (0,T)$ and if, for some $i\in\{1,\dots,d\}$, we have $\phi_i\in C^{1,2}(\mathbb R^N\times [0,T))$ such that
\begin{eqnarray*}
(i)&&\phi_i(x_0, t_0) = u_i^-(x_0,t_0),\notag\\
(ii)&&(x_0,t_0)\mbox{ is a global minimum of $\phi_i-u_i^-$},
\end{eqnarray*}
then
\begin{eqnarray*}
&&\min\biggl\{-\mathcal{H}^\kappa(\phi_i,u_i^-)-\psi_i, u_i^--\max_{j\neq i}(-c_{i,j}+u_j^-)\biggr\} \leq 0\mbox{ for all }\kappa\in (0,1).
\end{eqnarray*}
$(iii)$ If $(u_1,\dots,u_d)$ is both a viscosity  supersolution and subsolution to the problem in \eqref{eq4+}, then
 $(u_1,\dots,u_d)$ is a  viscosity solution to the problem in \eqref{eq4+}.
\end{definition}

\begin{remark} Note that it is natural, since the Levy measure is singular at $0$ and has some decay properties at infinity,  to break the non-local operator into the pieces $\mathcal{I}_\kappa$ and $\mathcal{I}^\kappa$ which give, respectively,
the contribution from the `small' and `large' jumps to $\mathcal{I}$. To have $\mathcal{I}_\kappa(x,t,v_i,D\phi_i)$  well defined the reasonable
things is to choose $v_i=\phi_i$ in Definition \ref{def:viscosity}. However, for $\mathcal{I}^\kappa(x,t,v_i,D\phi_i)$ to be well defined
the important thing is not regularity of $v_i$ but decay properties at infinity of $v_i$, decay properties which must be compatible with the decay properties of
the Levy measure $\nu$. Hence it is reasonable to use $\mathcal{I}^\kappa(x,t,v_i,D\phi_i)$ with $v_i=u_i^\pm$ in Definition \ref{def:viscosity} as long as
$u_i^\pm$ is assumed to have moderated growth at infinity. In our case polynomial growth is sufficient. We refer to \cite{JK06} for an elaboration on this definition and to \cite{BI08} for alternative equivalent definitions. %Euivalent through, e.g.Lemma 2.1. of BJK10
\end{remark}

\begin{remark}\label{yya} In the following we simply write $\mathcal{I}_\kappa(\phi,p)=\mathcal{I}_\kappa(x,t,\phi,p)$, $\mathcal{I}^\kappa(\phi,p)=\mathcal{I}^\kappa(x,t,\phi,p)$. Note that
\begin{eqnarray}\label{rule1}
K(x,t,z,\phi^1+\phi^2,p^1+p^2)=K(x,t,z,\phi^1,p^1)+K(x,t,z,\phi^2,p^2)
\end{eqnarray}
and hence that
\begin{eqnarray}\label{rule2}
I_\kappa(\phi^1+\phi^2,p^1+p^2)&=&I_\kappa(\phi^1,p^1)+I_\kappa(\phi^2,p^2),\notag\\
I^\kappa(\phi^1+\phi^2,p^1+p^2)&=&I^\kappa(\phi^1,p^1)+I^\kappa(\phi^2,p^2).
\end{eqnarray}
\end{remark}

\setcounter{equation}{0} \setcounter{theorem}{0} \setcounter{definition}{0}
\section{The comparison principle:  proof of Theorem \ref{Thm1-}}\label{sec:compprinciple}
\noindent
The purpose of this section is to prove Theorem \ref{Thm1-} and hence through out the section we adopt the assumption stated in
Theorem \ref{Thm1-}. In particular, let $\mathcal{H}$ be as in \eqref{opera} and assume \eqref{assump1el} and \eqref{assump1elel}. Assume that  $\psi_i$, $c_{i,j}$, and $g_i$ are as stated in Theorem \ref{Thm1-} and assume that $(u_1^-,\dots,u_d^-)$ and $(u_1^+,\dots,u_d^+)$ are viscosity sub- and supersolutions, respectively, to the problem in \eqref{eq4+}. We first prove the following lemma.
\begin{lemma}
 \label{lemma:supersol}  The following is true
 for any $\gamma> 0$. Let $\theta\geq 0$. Then there exists $\eta>0$, independent of $\theta$, such that if $\lambda\geq \eta$,
 then $(\bar{u}_1^+,\dots,\bar{u}_d^+)$, $$\bar{u}_i^+:=u_i^+ + \theta e^{-\lambda t}(|x|^{2\gamma+2}+1)\mbox{ for $i\in\{1,\dots,d\}$,}$$ is a viscosity supersolution of \eqref{eq4+}.
\end{lemma}
\begin{proof} Since $u_i^+ \in \mbox{LSC}_p(\mathbb R^N\times [0,T])$ we have $\bar u_i^+ \in \mbox{LSC}_p(\mathbb R^N\times [0,T])$.  Let
$(x_0,t_0)\in\mathbb R^N\times [0,T]$ and assume, for some $i\in\{1,\dots,d\}$, that $\phi_i\in C^{1,2}(\mathbb R^N\times [0,T))$,  satisfies
\begin{eqnarray*}
(i)&&\phi_i(x_0, t_0) = \bar u_i^+(x_0,t_0),\notag\\
(ii)&&(x_0,t_0)\mbox{ is a global maximum of $\phi_i-\bar u_i^+$}.
\end{eqnarray*}
To prove the lemma it is enough to prove that there exists $\eta>0$, independent of $\theta$, such that if $\lambda\geq \eta$, then
\begin{eqnarray*}
\min\biggl\{-\mathcal{H}^\kappa(\phi_i,\bar u_i^+)-\psi_i, \bar u_i^+-\max_{j\neq i}(-c_{i,j}+\bar u_j^+)\biggr\} \geq 0\mbox{ for all }\kappa\in (0,1).
\end{eqnarray*}
Let $\Phi_i=\phi_i-\theta e^{-\lambda t}(|x|^{2\gamma+2}+1)$ and note that
by construction $\Phi_i-u_i^+$ has a global maximum at $(x_0,t_0)$. Using that $u_i^+$ is a supersolution we have that
\begin{eqnarray}\label{eq4++hha+a}
\min\biggl\{-\mathcal{H}^\kappa(\Phi_i,u_i^+)-\psi_i, u_i^+-\max_{j\neq i}(-c_{i,j}+ u_j^+)\biggr\} \geq 0
\end{eqnarray}
for all $\kappa\in (0,1)$. Using  \eqref{eq4++hha+a} we see that
$$\bar u_i^+-\max_{j\neq i}(-c_{i,j}+ \bar u_j^+)= u_i^+-\max_{j\neq i}(-c_{i,j}+ u_j^+) \geq 0$$
 since $\bar{u}_i^+-u_i^+$ is independent of $i$. To conclude the proof we hence only have to ensure that
\begin{equation}\label{eq4++hha+ab}
-\mathcal{H}^\kappa(\phi_i,\bar u_i^+)-\psi_i\geq 0\mbox{ at $(x_0,t_0)$}\mbox{ and for all }\kappa\in (0,1).
\end{equation}
Let $\hat\phi(x,t)=\theta e^{-\lambda t}(|x|^{2\gamma+2}+1)$. Then $\Phi_i=\phi_i-\hat\phi$, $\bar u_i^+=u_i^++\hat\phi$.
To establish \eqref{eq4++hha+ab}  we first note, using \eqref{eq4++hha+a}, that, at $(x_0,t_0)$,
\begin{eqnarray}\label{ikk1}
0\leq -\mathcal{H}^\kappa(\Phi_i,u_i^+)-\psi_i&=&-\mathcal{L}\Phi_i-\mathcal{I}_\k(\Phi_i,D\Phi_i)-\mathcal{I}^\k(u_i^+,D\Phi_i)-\psi_i.
\end{eqnarray}
However, using Remark \ref{yya} we see that
\begin{eqnarray}\label{ikk2}
\mathcal{I}_\kappa(\Phi_i,D\Phi_i)&=&\mathcal{I}_\k(\phi_i,D\phi_i)+\mathcal{I}_\k(-\hat\phi,-D\hat\phi),\notag\\
\mathcal{I}^\k(u_i^+,D\Phi_i)&=&\mathcal{I}^\k(\bar u_i^+,D\phi_i)+\mathcal{I}^\kappa (-\hat\phi,-D\hat\phi).
\end{eqnarray}
Hence, combining \eqref{ikk1} and \eqref{ikk2} we can conclude that
\begin{eqnarray}\label{ikk3}
0&\leq&-\mathcal{L}\phi_i-\mathcal{I}_\k(\phi_i,D\phi_i)-\mathcal{I}^\k(\bar u_i^+,D\phi_i)-\psi_i\notag\\
&&+ \mathcal{L}\hat\phi-\mathcal{I}_\k(-\hat\phi,-D\hat\phi)-\mathcal{I}_\k(-\hat\phi,-D\hat\phi)\notag\\
&=&-\mathcal{H}^\kappa(\phi_i,\bar u_i^+)-\psi_i+\mathcal{L}\hat\phi-\mathcal{I}(-\hat\phi,-D\hat\phi).
\end{eqnarray}
In particular, we have
\begin{eqnarray*}
-\mathcal{H}^\kappa(\phi_i,\bar u_i^+)-\psi_i\geq -\mathcal{L}\hat\phi+\mathcal{I}(-\hat\phi,-D\hat\phi).
\end{eqnarray*}
Using this we first note that
\begin{eqnarray*}
-\mathcal{L}(\hat\phi)\geq\lambda\theta e ^{- \lambda t} (|x|^{2 \gamma +2}+1) -\mathcal{F}(\theta e ^{-\lambda t}( |x|^{2 \gamma +2}+1))
\end{eqnarray*}
where
\begin{equation*}
\mathcal{F} := \sum_{i,j=1}^N a_{ij}(x,t)\p_{x_{i}x_{j}}+ \sum_{i=1}^N a_{i}(x,t)\p_{x_{i}}.
\end{equation*}
Using the assumption on the operator $\mathcal{H}$ stated in Theorem \ref{Thm1-} we see that
\begin{eqnarray*}
|\mathcal{F}(\theta e ^{-\lambda t}( |x|^{2 \gamma +2}+1))|\leq
c\theta e ^{-\lambda t}(|x|^{2 \gamma +2}+1)
\end{eqnarray*}
and hence
\begin{equation}\label{eq4++hha+abdf}
-\mathcal{L}(\hat\phi)\geq (\lambda-c)\hat\phi.
\end{equation}
To estimate $\mathcal{I}(-\hat\phi,-D\hat\phi)$ we first note, by definition, that
\begin{eqnarray}\label{operanlockk1}
   \mathcal{I}(x,t,-\hat\phi,-D\hat\phi)&=&-\int_{\mathbb R^l\setminus\{0\}} K(x,t,z,\hat\phi,D\hat \phi)\nu\left(dz\right),\notag\\
   K(x,t,z,\hat\phi,D\hat \phi)&=&\hat\phi\left(x+\eta\left(x,t,z\right),t\right) -\hat\phi\left(x,t\right)\notag\\
    &&- \chi_{\{|z| \leq 1\}}(z) \sum_{k=1}^N\eta_k\left(x,t,z\right)\partial_{x_k}\hat\phi.
\end{eqnarray}
Hence,
\begin{eqnarray*}\label{operanlockk2}
   |\mathcal{I}(x,t,-\hat\phi,-D\hat\phi)|&\leq &\int_{B(0,1)\setminus\{0\}} |K(x,t,z,\hat\phi,D\hat \phi)|\nu\left(dz\right),\notag\\
   &&+\int_{\mathbb R^l\setminus B(0,1)} |K(x,t,z,\hat\phi,D\hat \phi)|\nu\left(dz\right)\notag\\
   &=:& T_1+T_2.
\end{eqnarray*}
Now, using \eqref{assump1elelnloc+}, and Taylor's formula,  we first see that
\begin{align}\label{operanlockk3}
 |K(x,t,z,\hat\phi,D\hat \phi)| &\leq c |z|^2 \hat\phi(x,t)&\mbox{whenever }&z\in B(0,1)\setminus\{0\}, \notag \\
 |K(x,t,z,\hat\phi,D\hat \phi)| &\leq c  \hat\phi(x,t) &\mbox{whenever }&z\in \mathbb R^l\setminus B(0,1).
\end{align}
Hence, using this, \eqref{assump1elelnloc}, and \eqref{assump1elelnloc+} we can conclude that
\begin{eqnarray*} %\label{operanlockk2}
T_1\leq c \hat\phi(x,t) \, \mbox{ and } \, T_2\leq c\hat\phi(x,t).
\end{eqnarray*}
Putting the estimates together we can conclude, at $(x_0,t_0)$ and for all $\kappa\in (0,1)$, that
\begin{eqnarray}\label{fin1}
-\mathcal{H}^\kappa(\phi_i,\bar u_i^+)-\psi_i&\geq& -\mathcal{L}\hat\phi+\mathcal{I}(-\hat\phi,-D\hat\phi) \geq (\lambda-c)\hat\phi(x_0,t_0).
\end{eqnarray}
In view of \eqref{eq4++hha+ab} we see that \eqref{fin1} completes the proof of the lemma.\end{proof}

\subsection{Proof of Theorem \ref{Thm1-}}
Assume that $(u_1^+,\dots,u_d^+)$ is a viscosity supersolution and that $(u_1^-,\dots,u_d^-)$ is a
viscosity subsolution, respectively, to the problem in \eqref{eq4+}. We want to prove that
 \begin{eqnarray} \label{comp1}
 u_i^-(x,t)\leq u_i^+(x,t),\mbox{ for all } i\in\{1,\dots,d\},
\end{eqnarray}
whenever $(x,t)\in \mathbb R^N\times (0,T]$. In fact, we will prove a slightly modified version of \eqref{comp1}. We let $\tilde u_i^-(x,t)=e^tu_i^-(x,t)$ and $\tilde u_i^+(x,t)=e^t\bar u_i^+(x,t)$
for all  $i\in\{1,\dots,d\}$. One can easily verify that $(\tilde u_1^-,\dots,\tilde u_d^-)$ is a viscosity subsolution to the problem
\begin{eqnarray} \label{eq:systemgamma}
&&\min\left \{\tilde u_i^- -\mathcal{H}\tilde u_i^-(x,t)- \tilde \psi_i(x,t),\tilde u_i^-(x,t)-\max_{j\neq i}(-\tilde c_{i,j}(x,t)+\tilde u_j^-(x,t))\right \}=0,\notag\\
&&\ \tilde u_i^-(x,T)=\tilde g_i(x),
\end{eqnarray}
where
$$\tilde\psi_i(x,t)=e^t\psi_i(x,t),\ \tilde c_{i,j}(x,t)=e^tc_{i,j}(x,t),\ \tilde g_i(x)=e^T g_i(x),
$$
if and only if $(u_1^-,\dots,u_d^-)$ is a subsolution to \eqref{eq4+}. Similarly,
$(\tilde u_1^+,\dots,\tilde u_d^+)$ is a viscosity supersolution to the problem in \eqref{eq:systemgamma} if and only if
 $(u_1^+,\dots,u_d^+)$ is a supersolution to \eqref{eq4+}. Proving \eqref{comp1} is now equivalent to proving
 \begin{eqnarray} \label{comp2+}
 \tilde u_i^-(x,t)\leq \tilde u_i^+(x,t),\mbox{ for all } i\in\{1,\dots,d\},
\end{eqnarray}
whenever $(x,t)\in \mathbb R^N\times (0,T]$. However, according to Lemma \ref{lemma:supersol} it is enough to prove
 \begin{eqnarray}\label{comp2-}
 \tilde u_i^-(x,t)\leq \bar u_i^+(x,t),\mbox{ for all } i\in\{1,\dots,d\},
\end{eqnarray}
whenever $(x,t)\in \mathbb R^N\times (0,T]$, where $\bar{u}_i^+:=\tilde u_i^+ + \theta e^{-(\lambda-1) t}(|x|^{2\gamma+2}+1)$, $\gamma >0$,  since we easily recover \eqref{comp2+} by letting $\theta\to 0$ in  \eqref{comp2-}. In fact, it is enough to prove that
\begin{equation}\label{comp2}
\tilde u_i^-(x,t)\leq \bar u_i^+(x,t) + \frac \theta t,\mbox{ for all } i\in\{1,\dots,d\},
\end{equation}
whenever $(x,t)\in\mathbb R^N\times(0,T]$ and for any $\theta >0$, since the desired result is still retrieved in the limit as $\theta \to 0$. In particular, we will prove \eqref{comp2} for $\theta$ fixed and then let $\theta\to 0$.

Let in the following $B(0,R)$, $R>0$, be the standard Euclidean ball of radius $R$ centered at 0.  Using that
$u_i^-\in\mbox{USC}_{p}(\mathbb R^N\times [0,T])$ for $i\in\{1,\dots,d\}$, we have that there exists $\gamma_0 \geq \frac{1}{2}$ such that $|\tilde u_i^-(x,t)|\leq c(1+|x|^{2\gamma_0})$. We now fix this $\gamma_0$ and plug it into our defintion of $\bar u_i^+(x,t)$ from above. With slight abuse of notation, we will drop the subscript from $\gamma_0$ and only write $\gamma$ and let $\bar u_i^+(x,t)$ from here on in denote $\tilde u_i^+ + \theta e^{-(\lambda-1) t}(|x|^{2\gamma+2}+1)$ with  $\gamma=\gamma_0$ fixed. Using this $\tilde u_i^+$ we see that there exists $R>0$ such that
 \begin{eqnarray} \label{comp4}
\tilde u_i^-(x,t)-\bar u_i^+(x,t) <0\mbox{ whenever $(x,t)\in (\mathbb R^N\setminus B(0,R))\times[0,T]$}.
\end{eqnarray}
From \eqref{comp4} we see that
 \begin{eqnarray} \label{comp5}
&&\sup_{(x,t)\in \mathbb R^N\times[0,T]}\max_{i\in\{1,\dots,d\}}((\tilde u_i^-(x,t)-\bar u_i^+(x,t) - \frac \theta t)\notag\\
&=&\sup_{(x,t)\in B(0,R)\times[0,T]}\max_{i\in\{1,\dots,d\}}((\tilde u_i^-(x,t)-\bar u_i^+(x,t)- \frac \theta t)\notag\\
&=&\max_{i\in\{1,\dots,d\}}((\tilde u_i^-(\bar  x,\bar  t)-\bar u_i^+(\bar  x,\bar  t) - \frac {\theta} {\bar t})
\end{eqnarray}
for some $(\bar  x,\bar  t)\in B(0,R)\times(0,T]$. We will now prove \eqref{comp2} by contradiction. Indeed,  assume that
 \begin{equation} \label{comp3}
 \max_{i\in\{1,\dots,d\}}(\tilde u_i^-(\bar x,\bar t)-\bar u_i^+(\bar x,\bar t)- \frac {\theta} {\bar t})>0.
\end{equation}
Using that $\tilde u_i^-(x,T)\leq \bar u_i^+(x,T)$, for all $x\in\mathbb R^N$,  by the definition of sub- and supersolution, we see that
$(\bar  x,\bar  t)\in B(0,R)\times(0,T)$. For $(\bar  x,\bar  t)\in B(0,R)\times(0,T)$ fixed, we let $\mathcal{I}$ be the non-empty set of
all $j \in\{1,\dots,d\}$ such that
\begin{equation*}%\label{comp6}
 (\tilde u_j^-(\bar  x,\bar  t) -\bar u_j^+(\bar  x,\bar  t)- \frac{\theta}{\bar t})=\max_{i\in\{1,\dots,d\}}(\tilde u_i^-(\bar  x,\bar  t)-\bar u_i^+(\bar  x,\bar  t) -\frac{\theta}{\bar t}).
\end{equation*}
For $\gamma=1$ and for given degrees of freedom $\beta>0$, $\Lambda> 0$, we introduce the function $\varphi_\e: \R^N \times \R^N \times [0,T] \to  \R$,
\begin{equation*} %\label{comp8}
\varphi_\e(x,y,t) = \frac{1}{2\e}|x-y|^{2} + \Lambda  (|x-\bar  {x}|^{2\gamma +2} +|y-\bar  {x}|^{2\gamma +2}) + \beta(t-\bar  {t})^2.
\end{equation*}
Note that $\varphi_\e$ is non-negative.  Furthermore, for $j\in \mathcal{I}$ fixed, and $\e$, $0<\e\ll 1$,  we consider the function
\begin{equation*} \label{comp7}
\phi^j_\e(x,y,t) = \tilde u_j^-(x,t) - \bar u_j^+(y,t) - \frac{\theta}{t}- \varphi_\e(x,y,t),
\end{equation*}
where $(x,y,t) \in  B(0,R) \times B(0,R)\times [0,T]$. Using that $\tilde u_j^-$ is upper semi-continuous and that
$\bar u_j^+$ is lower semi-continuous we can conclude that $\phi^j_\e$ is upper semi-continuous and hence that
there exist $(x_\e, y_\e,t_\e)$ such that
\begin{equation}\label{comp9}
\phi_\e ^j (x_\e, y_\e, t_\e) = \sup_{(x,y,t) \in  \overline{B(0,R)} \times \overline{B(0,R)} \times [0,T]} \phi^j_\e(x,y,t).
\end{equation}
Note that the points $(x_\e, y_\e, t_\e)$ depend also on $\beta$ and $\Lambda$. However, in this part of the argument
$\beta$ and $\Lambda$ are kept fixed and hence the dependence is harmless in the following.
Using that  $2\phi^j(x_\e,y_\e,t_\e) \geq \phi^j(x_\e,x_\e,t_\e)+\phi^j(y_\e,y_\e,t_\e)$ we see that
\begin{equation}\label{comp10}
\frac{1}{\e} |x_\e-y_\e|^{2}  \leq \tilde u_j^-(x_\e,t_\e) - \tilde u_j^-(y_\e,t_\e) + \bar u_j^+(x_\e,t_\e) - \bar u_j^+(y_\e,t_\e),
\end{equation}
and as the right hand side of \eqref{comp10} is bounded we have that $|x_\e-y_\e|\to 0$ as $\e\to 0$.
Using that
$(\bar  x,\bar  t)\in \overline{B(0,R)}\times(0,T)$, and the construction of $\varphi _\e$, we see that
\begin{eqnarray}\label{ffa}
\tilde u_j^-(\bar{x},\bar{t})-\bar u_j^+(\bar{x},\bar{t}) -\frac{\theta}{\bar t} &=& \phi^j_\e (\bar{x},\bar{x},\bar{t}) \notag\\
&\leq &\phi^j_\e(x_\e , y_\e,t_\e)\leq \tilde u_j^-(x_\e,t_\e)-\bar u_j^+(y_\e,t_\e) -\frac{\theta}{t_\e}.
\end{eqnarray}
Furthermore, we see that we must have, using the definition of $(\bar  x,\bar  t)$, $(x_\e , y_\e,t_\e)$, and the upper semi-continuity of
$\tilde u^-_j-\bar u^+_j$, that
\begin{equation}\label{ffa1}
(x_\e,y_\e,t_\e) \to (\bar{x},\bar{x},\bar t)\mbox{ as $\e\to 0$}.
\end{equation}
The above display also shows that, for $\e$ small enough, we have $t_\e \in (0,T)$ since $t_\e \to \bar t$ and $\bar t  \in (0,T)$.
Note also that
\begin{equation}\label{eq:tbarxbarconv}
\tilde u_j^-(x_\e,t_\e) \to \tilde u_j^-(\bar{x},\bar t)\mbox{ and } \bar u_j^+(y_\e,t_\e) \to \bar  u_j^+(\bar{x},\bar t)\mbox{ as $\e\to 0$}.
\end{equation}
Indeed, recall that $\tilde u_j^-$ is  upper semi-continuous and assume, taking \eqref{ffa1} into account, that $\limsup_{\e\to 0}\tilde u_j^-(x_\e,t_\e) < \tilde u_j^-(\bar{x},\bar t)$. Then, using
\eqref{ffa} we have that $\liminf_{\e\to 0}\bar u_j^+(y_\e,t_\e) < \bar u_j^+(\bar{x},\bar{t})$ but this contradicts the lower semi-continuity for $\tilde u_j^+$. Similarly, assuming that $\liminf_{\e\to 0}\bar u_j^+(y_\e,t_\e) >\bar u_j^+(\bar{x},\bar{t})$ we see that
$$\limsup_{\e\to 0}\tilde u_j^-(x_\e,t_\e) > \tilde u_j^-(\bar{x},\bar t),$$
which again is a contradiction. Repeating \eqref{ffa} we also have that
\begin{eqnarray}\label{ffahha}
&&\tilde u_j^-(\bar{x},\bar{t})-\bar u_j^+(\bar{x},\bar{t}) - \frac{\theta}{\bar t} = \phi^j_\e (\bar{x},\bar{x},\bar{t})\notag\\
&\leq& \phi^j_\e(x_\e , y_\e,t_\e) = \tilde u_j^-(x_\e,t_\e)-\bar u_j^+(y_\e,t_\e)-\varphi_\e(x_\e , y_\e,t_\e) -\frac{\theta}{t_\e}.
\end{eqnarray}
In particular,
\begin{eqnarray}\label{ffahha+}
\varphi_\e(x_\e , y_\e,t_\e)&\leq& \tilde u_j^-(x_\e,t_\e)-\tilde u_j^-(\bar{x},\bar{t})\notag\\
&&+\bar u_j^+(\bar{x},\bar{t})-\bar u_j^+(y_\e,t_\e) + \frac{\theta}{\bar t} -\frac{\theta}{t_\e}
\end{eqnarray}
and using \eqref{eq:tbarxbarconv} we see that
\begin{eqnarray}\label{ffahha+a}
\lim_{\e\to 0}\varphi_\e(x_\e , y_\e,t_\e)=0.
\end{eqnarray}
In particular,
\begin{equation}\label{ffahha+any}
\frac{1}{\e} |x_\e-y_\e|^{2}  \to 0\mbox{ as $\e \to 0$.}
\end{equation}

To proceed we will now argue as in \cite{IK91}, using the no-loop condition \eqref{assump3} $(ii)$, to conclude that
there exists $k \in \mathcal{I}$ such that
\begin{equation}
\label{eq:greaterclaim}
\tilde u_k^-(\bar{x},\bar{t}) > \max_{i \in \{1,\dots,k-1,k+1,\dots,d\}} (\tilde u_i^-(\bar{x},\bar{t}) -\tilde c_{k,i}(\bar{x}, \bar{t})).
\end{equation}
Indeed, assume, on the contrary, that
$$
\tilde u_k^-(\bar{x},\bar{t}) \leq \max_{i \in \{1,\dots,k-1,k+1,\dots,d\}} (\tilde u_i^-(\bar{x},\bar{t}) -\tilde c_{k,i}(\bar{x}, \bar{t}))
$$
for all $k\in \mathcal{I}$ and hence, in particular, that
$$\tilde u_k^-(\bar{x},\bar{t}) +  \tilde c_{k,j}(\bar{x}, \bar{t}) \leq \tilde u_j^-(\bar{x},\bar{t})$$
for some $j \in \{1,\dots,k-1,k+1,\dots,d\}$. Furthermore, since $(\bar u_1^+,\dots,\bar u_d^+)$ is a supersolution to  \eqref{eq:systemgamma} we have that
\begin{equation*}
\bar u_k^+(\bar{x},\bar{t}) \geq  \bar u_j^+(\bar{x},\bar{t})- \tilde c_{k,j}(\bar{x},\bar{t}).
\end{equation*}
Combining the two inequalities above yields
\begin{equation}
\tilde u_k^-(\bar{x},\bar{t}) -\tilde u_j^-(\bar{x},\bar{t}) \leq -  \tilde c_{k,j}(\bar{x}, \bar{t}) \leq \bar u_k^+(\bar{x},\bar{t}) -  \bar u_j^+(\bar{x},\bar{t}),
\end{equation}
and hence
\begin{equation}
\label{eq:switchequal+}
\tilde u_k^-(\bar{x},\bar{t}) - \bar u_k^+(\bar{x},\bar{t})  - \frac {\theta}{\bar t}\leq  \tilde u_j^-(\bar{x},\bar{t}) -  \bar u_j^+(\bar{x},\bar{t}) - \frac {\theta}{\bar t}.
\end{equation}
But $k\in \mathcal{I}$ so \eqref{eq:switchequal+} is actually an equality and hence $j\in \mathcal {I}$. Repeating this argument as many times as necessary we get the existence of a loop of indices $\{i_1, i_2, \dots,i_p,i_{p+1}\}$ such that $i_1=i_{p+1}$ and
\begin{equation}
\tilde c_{i_1,i_2}+\tilde c_{i_2,i_3} + \dots + \tilde c_{i_p,i_{p+1}} =0.
\end{equation}
This contradicts our assumptions on the switching costs and hence \eqref{eq:greaterclaim} must hold.

 We now consider $k\in \mathcal{I}$ such that \eqref{eq:greaterclaim} holds and we intend to derive a contradiction to the assumption in
 \eqref{comp3}. First, using \eqref{eq:tbarxbarconv} and \eqref{eq:greaterclaim} we see that there exists $\bar\e$, $0<\bar\e\ll 1$, such that
 \begin{equation}\label{appa}
\tilde u_k^-(x_\e, t_\e) > \max_{i \in \{1,\dots,k-1,k+1,\dots,d\}} (\tilde u_i^-(x_\e,t_\e) - \tilde c_{k,i}(x_\e,t_\e))\mbox{ when $0<\e<\bar\e$}.
\end{equation}
\eqref{appa} ensures that $\tilde u_k^-$ is above the obstacle at the points $\{(x_\e,t_\e)\}_{\e<\bar\e}$.  We next intend to apply the so called maximum principle for semi-continuous functions, in our case adapted to the non-local system, in a neighborhood of $(x_\e, y_\e,t_\e)$ and
 for $\tilde u_k^-$, and along the lines of \cite{BI08}. To do this we first have to calculate $\partial_t\varphi_\e$, $\partial_{x_i}\varphi_\e$, $\partial_{y_i}\varphi_\e$ and $\partial_{x_iy_j}\varphi_\e$.  Doing this we see that
\begin{eqnarray}\label{eq:derivatives}
&&\partial_t\varphi_\e(x,y,t)=2 \beta(t- \bar{t}), \notag \\
&&D_x\varphi_\e(x,y,t)= \frac{1}{\e}(x-y)+  \Lambda (2\gamma +2) (x-\bar{x}) |x-\bar{x}|^{2\gamma}, \notag\\
&&D_y\varphi_\e(x,y,t) =-\frac{1}{\e}(x-y)+ \Lambda (2\gamma +2) (y-\bar{x}) |y-\bar{x}|^{2\gamma}.
\end{eqnarray}
Furthermore,
\begin{eqnarray}\label{eq:derivatives+}
D^2_{x,y}\varphi_\e (x,y)= \frac{1}{\e} \left( \begin{array}{cc} I & -I \\ -I & I  \end{array} \right) + \left(\begin{array}{cc} l(x) & 0 \\ 0 & l(y)  \end{array}\right)
\end{eqnarray}
where
\begin{eqnarray*}%\label{eq:derivatives+a}
%l_1(x,y)&=& 1 |x-y|^{2\gamma-2} I + \gamma(2\gamma-2)(x-y)(x-y)^\ast|x-y|^{2\gamma-4}, \notag\\
l(x)&=& \Lambda(2\gamma+2) |x-\bar{x}|^{2\gamma} I +2\gamma\Lambda (2\gamma+2)\langle x-\bar{x},x-\bar{x}\rangle|x-\bar{x}|^{2\gamma-2}.
\end{eqnarray*}
Let $S_N$ be the set of all $N\times N$-dimensional symmetric matrices
 and let $ \varphi_\e^\alpha$ be the $\sup$-convolution of $\varphi_\e$ as defined in \cite{BI08}.
%Note that $\tilde u^-, -\bar u^+ \leq c(1+ |x|^2)$ and that $\varphi_\e \in C^{1,2}(\R^N \times \R^N \times (0,T))$. Hence, we can apply Theorem 2.2 of \cite{JK05}
We may now apply Lemma $1$ of \cite{BI08} (more precisely Corollary 2) to conclude that for any $\kappa >0$ there exists $\bar \alpha (\kappa)$ such that for $0 < \alpha \leq \bar{\alpha}$ there exist $C,D \in \R$ and $X,Y \in S_N$, such that
\begin{eqnarray*}%\label{eq:hessianest}
C+D = 2 \beta(t_\e -\bar{t}) - \frac {\theta}{t_\e^2},
\end{eqnarray*}
\begin{eqnarray}\label{eq:hessianest+}
\left(\begin{array}{cc}
X & 0 \\
0 & -Y
\end{array} \right)\leq %\frac{2}{\e}\left( \begin{array}{cc} I & -I \\ -I & I  \end{array} \right) + \left(\begin{array}{cc} l(x_\e) & 0 \\ 0 & l(y_\e) \end{array}\right)
D^2_{x,y}\varphi_\e(x_\e,y_\e) + o_{\alpha}(1)\left(\begin{array}{cc} I & 0 \\ 0 & I \end{array}\right),
\end{eqnarray}
and such that 
\begin{eqnarray}\label{subb}
&&-C + \tilde u_k^-(x_\e,t_\e) -\sum_{i,j=1}^N a_{ij}(x_\e,t_\e)X_{i,j} -\sum_{i=1}^N a_i(x_\e,t_\e)\partial_{x_i}\varphi_\e(x_\e,y_\e,t_\e)\notag\\
 &&- \mathcal{I}_\kappa(x_\e,t_\e,\varphi^\alpha_\e(\cdot,y_\e,t_\e),D_x\varphi_\e)\notag\\
 &&-\mathcal{I}^\kappa(x_\e,t_\e,\tilde u_k^-(\cdot,t_\e),D_x\varphi_\e)
 -\tilde \psi_k(x_\e,t_\e) \leq 0,
\end{eqnarray}
and
\begin{eqnarray}\label{subb+}
&&D +\bar u_k^+(y_\e, t_\e) -\sum_{i,j=1}^N a_{ij}(y_\e,t_\e)Y_{i,j}+\sum_{i=1}^N a_i(y_\e,t_\e)\partial_{y_i}\varphi_\e(x_\e,y_\e,t_\e)\notag\\
 &&- \mathcal{I}_\kappa(y_\e,t_\e,-\varphi^\alpha_\e(x_\e,\cdot,t_\e),-D_y\varphi_\e)\notag\\
 &&-\mathcal{I}^\kappa(y_\e,t_\e,\bar u_k^+(\cdot,t_\e),-D_y\varphi_\e)-\tilde \psi_k(y_\e,t_\e) \geq 0.
\end{eqnarray}
In the above, we have used little o notation, i.e., $o_\alpha(1) \to 0$ as $\alpha \to 0$. Adding \eqref{subb} and \eqref{subb+} we see that
\begin{eqnarray}\label{subb++}
&&-C-D + \tilde u_k^-(x_\e,t_\e)-\bar u_k^+(y_\e, t_\e) -\sum_{i,j=1}^N (a_{ij}(x_\e,t_\e)X_{i,j}-a_{ij}(y_\e,t_\e)Y_{i,j})\notag\\
&&-\sum_{i=1}^N (a_i(x_\e,t_\e)\partial_{x_i}\varphi_\e(x_\e,y_\e,t_\e)+a_i(y_\e,t_\e)\partial_{y_i}\varphi_\e(x_\e,y_\e,t_\e))\notag\\
&&-(\tilde \psi_k(x_\e,t_\e) - \tilde \psi_k(y_\e,t_\e)) \notag\\
&&-\left (\mathcal{I}_\kappa(x_\e,t_\e,\varphi^\alpha_\e(\cdot,y_\e,t_\e),D_x\varphi_\e)
-\mathcal{I}_\kappa(x_\e,t_\e,-\varphi^\alpha_\e(x_\e,\cdot,t_\e),-D_y\varphi_\e)\right )\notag\\
&&-\left(\mathcal{I}^\kappa(y_\e,t_\e,\tilde u_k^-(\cdot,t_\e),D_x\varphi_\e)-\mathcal{I}^\kappa(y_\e,t_\e,\bar u_k^+(\cdot,t_\e),-D_y\varphi_\e) \right)\leq 0.
\end{eqnarray}
Now, using standard arguments based on assumptions \eqref{assump1el} and \eqref{assump1elel} it follows that
\begin{eqnarray*}
&&\left |\sum_{i=1}^N \left (a_i(x_\e,t_\e)\partial_{x_i}\varphi_\e(x_\e,y_\e,t_\e)+a_i(y_\e,t_\e)\partial_{y_i}\varphi_\e(x_\e,y_\e,t_\e) \right) \right|\notag\\
&\leq& c \bigg  (\frac 1\e |x_\e-y_\e|^{2}+\Lambda \left  (1+|x_\e||x_\e-\bar{x}|^{2\gamma+1}+|y_\e| |y_\e-\bar{x}|^{2\gamma+1} \right)\bigg  )\end{eqnarray*}
and
\begin{eqnarray*}
&&\left |\sum_{i,j=1}^N \left (a_{ij}(x_\e,t_\e)X_{i,j} - a_{ij}(y_\e,t_\e)Y_{i,j} \right ) \right | \notag \\
 &\leq& c \bigg (\frac{1}{\e}|x_\e-y_\e|^{2}  +\Lambda  (1+|x_\e|^2|x_\e-\bar x|^{2 \gamma} + |y_\e|^2|y_\e- \bar x|^{2\gamma }) \bigg ).\notag
\end{eqnarray*}
Putting these estimates together we find that
\begin{eqnarray*}%\label{eq:est}
&&-C -D + \tilde u_k^-(x_\e,t_\e) -\bar u_k^+(y_\e, t_\e)\notag\\
&&-(\mathcal{I}_\kappa(x_\e,t_\e,\varphi^\alpha_\e(\cdot,y_\e,t_\e),D_x\varphi_\e)
-\mathcal{I}_\kappa(x_\e,t_\e,-\varphi^\alpha_\e(x_\e,\cdot,t_\e),-D_y\varphi_\e))\notag\\
&&-(\mathcal{I}^\kappa(y_\e,t_\e,\tilde u_k^-(\cdot,t_\e),D_x\varphi_\e)-\mathcal{I}^\kappa(y_\e,t_\e,\bar u_k^+(\cdot,t_\e),-D_y\varphi_\e))\notag\\
&\leq&  \frac{c}{\e}|x_\e-y_\e|^{2}+
c \Lambda h(x_\e, y_\e,\bar x,\bar y)+ \tilde \psi_k(x_\e,t_\e) - \tilde \psi_k(y_\e,t_\e)
\end{eqnarray*}
where $0\leq h(x_\e, y_\e,\bar x,\bar y)\leq c$ for all $\e$, $0<\e\leq\bar\e$.
Hence, using the relation for $C+D$, see \eqref{eq:hessianest+} we see that
\begin{eqnarray*}
&&\tilde u_k^-(x_\e,t_\e) -\bar  u_k^+(y_\e, t_\e)-2 \beta(t_\e -\bar{t}) + \frac{\theta}{t_\e^2}\notag\\
&&-(\mathcal{I}_\kappa(x_\e,t_\e,\varphi^\alpha_\e(\cdot,y_\e,t_\e),D_x\varphi_\e)
-\mathcal{I}_\kappa(x_\e,t_\e,-\varphi^\alpha_\e(x_\e,\cdot,t_\e),-D_y\varphi_\e))\notag\\
&&-(\mathcal{I}^\kappa(y_\e,t_\e,\tilde u_k^-(\cdot,t_\e),D_x\varphi_\e)-\mathcal{I}^\kappa(y_\e,t_\e,\bar u_k^+(\cdot,t_\e),-D_y\varphi_\e))\notag\\
&\leq&  \frac{c}{\e}|x_\e-y_\e|^{2}+
c \Lambda h(x_\e, y_\e,\bar x,\bar y)+ \tilde \psi_k(x_\e,t_\e) - \tilde \psi_k(y_\e,t_\e).
\end{eqnarray*}
We now need to estimate the expressions involving the non-local operators and we intend to prove that
 \begin{eqnarray}\label{estnloc}
&&\biggl |-(\mathcal{I}_\kappa(x_\e,t_\e,\varphi^\alpha_\e(\cdot,y_\e,t_\e),D_x\varphi_\e)
-\mathcal{I}_\kappa(x_\e,t_\e,-\varphi^\alpha_\e(x_\e,\cdot,t_\e),-D_y\varphi_\e))\notag\\
&&-(\mathcal{I}^\kappa(y_\e,t_\e,\tilde u_k^-(\cdot,t_\e),D_x\varphi_\e)-\mathcal{I}^\kappa(y_\e,t_\e,\bar u_k^+(\cdot,t_\e),-D_y\varphi_\e))\biggr |\notag\\
&\leq&  \frac{c}{\e}|x_\e-y_\e|^{2}+
c \Lambda \hat h(x_\e, y_\e,\bar x,\bar y) +\frac {f(\kappa)}\e + o_\alpha(1),
\end{eqnarray}
where $0\leq \hat h(x_\e, y_\e,\bar x,\bar y)\leq c$ for all $\e$, $0<\e\leq\bar\e$,
and for some non-negative function $f(\kappa)$ such that  $f(\kappa)\to 0$ as $\kappa\to 0$. Note also that the points $\{(x_\e,y_\e,t_\e)\}_{\e<\bar\e}$ are independent of $\kappa$ and that the function $f(\kappa)$ is independent of $\e$. Assuming that \eqref{estnloc} holds and combining the estimates in the last displays we see that
\begin{eqnarray*}
&&\tilde u_k^-(x_\e,t_\e) -\bar  u_k^+(y_\e, t_\e) + \frac{\theta}{t_\e^2}\notag\\
&\leq&2 \beta(t_\e -\bar{t})+ \frac{c}{\e}|x_\e-y_\e|^{2}+
c \Lambda h(x_\e, y_\e,\bar x,\bar y)+
c \Lambda \hat h(x_\e, y_\e,\bar x,\bar y)\notag\\
&&+ \tilde \psi_k(x_\e,t_\e) - \tilde \psi_k(y_\e,t_\e)+\frac {f(\kappa)}\e + o_{\alpha}(1).
\end{eqnarray*}
Now, letting first $\alpha \to 0$ and then $\kappa\to 0$ we get
\begin{eqnarray*}
&&\tilde u_k^-(x_\e,t_\e) -\bar  u_k^+(y_\e, t_\e) + \frac{\theta}{t_\e^2}\notag\\
&\leq&2 \beta(t_\e -\bar{t})+ \frac{c}{\e}|x_\e-y_\e|^{2}+
c \Lambda h(x_\e, y_\e,\bar x,\bar y)+
c \Lambda \hat h(x_\e, y_\e,\bar x,\bar y)\notag\\
&&+ \tilde \psi_k(x_\e,t_\e) - \tilde \psi_k(y_\e,t_\e).
\end{eqnarray*}
Then, letting  $\e\to 0$, using \eqref{ffa1}, \eqref{ffahha+any} and the continuity of $\tilde \psi_k$, we can conclude that
\begin{eqnarray*}
&&\tilde u_k^-(\bar{x},\bar t) -\bar u_k^+(\bar{x},\bar t) + \frac {\theta}{\bar t^2}\leq \tilde c\Lambda.
\end{eqnarray*}
Finally,  letting $\Lambda \to 0$ in the last display we can conclude that
\begin{equation}
  \tilde u_k^-(\bar{x},\bar t) -\bar u_k^+(\bar{x},\bar t) + \frac {\theta}{\bar t^2} \leq 0,
\end{equation}
which contradicts \eqref{comp3}.

To finish the proof we now only have to prove the estimate in \eqref{estnloc}. However, by following section 5 of \cite{BI08}, we see that we may replace $\mathcal{I}_\kappa(x_\e,t_\e,\varphi^\alpha_\e(\cdot,y_\e,t_\e),D_x\varphi_\e)$ by $\mathcal{I}_\kappa(x_\e,t_\e,\varphi_\e(\cdot,y_\e,t_\e),D_x\varphi_\e) + o_{\alpha}(1)$ and hence it suffices to prove \eqref{estnloc} with $\varphi^\alpha_\e$ replaced by $\varphi_\e$. To do this we first see, using Taylor's formula, the definition of $\varphi_\e$, \eqref{assump1elelnloc} and \eqref{assump1elelnloc+}, that
\begin{align}\label{eq:estkappa0}
&|(\mathcal{I}_\kappa(x_\e,t_\e,\varphi_\e(\cdot,y_\e,t_\e),D_x\varphi_\e) \notag \\
-&\mathcal{I}_\kappa(y_\e,t_\e,-\varphi_\e(x_\e,\cdot,t_\e),-D_y\varphi_\e))|\leq \frac {f(\kappa)}\e
\end{align}
for all $\e$, $0<\e\leq\bar\e$, and for some non-negative function $f(\kappa)$ as above.
By the dominated convergence theorem we also have $f(\kappa)\to 0$ as $\kappa\to 0$ since
 \begin{eqnarray*}\lim _{\kappa \to 0} \int \limits_ {B(0,\kappa)} |z|^2 \nu(dz) &=& \lim _{\kappa \to 0 } \int \limits_ {B(0,1)} |z|^2 \chi_{\{|z|\leq \kappa\}}\nu(dz)\notag\\
 & =& \int \limits_ {B(0,1)} \lim _{\kappa \to 0} \left(|z|^2 \chi_{\{|z|\leq \kappa\}} \right )\nu(dz) =0.
 \end{eqnarray*}
Hence, only the term involving $\mathcal{I}^\kappa$ remains. 
To conduct estimates we first decompose
\begin{eqnarray*}%\label{eq:est}
&&-(\mathcal{I}^\kappa(x_\e,t_\e,\tilde u_k^-(\cdot,t_\e),D_x\varphi_\e)-\mathcal{I}^\kappa(y_\e,t_\e,\bar u_k^+(\cdot,t_\e),-D_y\varphi_\e))\notag\\
&=&-(\mathcal{I}_-^\kappa(x_\e,t_\e,\tilde u_k^-(\cdot,t_\e),D_x\varphi_\e)-\mathcal{I}_-^\kappa(y_\e,t_\e,\bar u_k^+(\cdot,t_\e),-D_y\varphi_\e))\notag\\
&&-(\mathcal{I}^\kappa_+(x_\e,t_\e,\tilde u_k^-(\cdot,t_\e),D_x\varphi_\e)-\mathcal{I}_+^\kappa(y_\e,t_\e,\bar u_k^+(\cdot,t_\e),-D_y\varphi_\e))
\end{eqnarray*}
where $\mathcal{I}_-^\kappa$ indicates integration with respect to $z\in B(0,1)\setminus B(0,\kappa)$ and $\mathcal{I}_+^\kappa$ indicates integration
 with respect to $z\in \mathbb R^l\setminus B(0,1)$.  Note that by construction, see \eqref{comp9}, we can, working with index $k$, assume that the sequence $\{(x_\e,y_\e,t_\e)\}_{\e<\bar\e}$ is such that
\begin{eqnarray}\label{comp9uu}
\phi_\e^k (x_\e, y_\e, t_\e) &=& \sup_{(x,y,t) \in  \overline{B(0,R)} \times \overline{B(0,R)} \times (0,T)} \phi^k_\e(x,y,t),\mbox{ where}\notag\\
\phi^k_\e(x,y,t) &=& \tilde u_k^-(x,t) - \bar u_k^+(y,t) - \frac{\theta}{t}- \varphi_\e(x,y,t).
\end{eqnarray}
To proceed we exploit that the maximizing property of the sequence $\{(x_\e,y_\e,t_\e)\}_{\e<\bar\e}$ in \eqref{comp9uu} implies, for $z\in \R^l \setminus B(0,\kappa)$, that
\begin{eqnarray}\label{comp9uu+}
 &&\tilde u_k^-(x_\e+\eta(x_\e,t_\e,z),t_\e) -\tilde u_k^-(x_\e,t_\e)-\sum_l\eta_l(x_\e,t_\e,z)\partial_{x_l}\varphi_\e(x_\e,y_\e,t_\e)\notag\\
 &&-\bar u_k^+(y_\e+\eta(y_\e,t_\e,z),t_\e) +\bar u_k^+(y_\e,t_\e)-\sum_l\eta_l(y_\e,t_\e,z)\partial_{y_l}\varphi_\e(x_\e,y_\e,t_\e)\notag\\
 &\leq &  \varphi_\e(x_\e+\eta(x_\e,t_\e,z),y_\e+\eta(y_\e,t_\e,z),t_\e)-\varphi_\e(x_\e,y_\e,t_\e)\notag\\
 &&-\sum_l\eta_l(x_\e,t_\e,z)\partial_{x_l}\varphi_\e(x_\e,y_\e,t_\e)-\sum_l\eta_l(y_\e,t_\e,z)\partial_{y_l}\varphi_\e(x_\e,y_\e,t_\e).
\end{eqnarray}
%To estimate the expression involving $\mathcal{I}_-^\kappa$ we have to, in particular, recover a factor $|z|^2$.  To estimate the expression involving $\mathcal{I}_+^\kappa$ should be easier than estimating the expression involving $\mathcal{I}_-^\kappa$.
To ease the notation in the calculations to follow we let $\eta_x^\e(z)=\eta(x_\e,t_\e,z)$, $\eta_y^\e(z)=\eta(y_\e,t_\e,z)$, and we let $\lan \cdot, \cdot \ran$ denote that standard scalar product in $\mathbb R^N$. To proceed we first note here that the left hand side of \eqref{comp9uu+} is precisely the integrand of
\begin{eqnarray}\label{eq:needestpp}
(\mathcal{I}^\kappa(x_\e,t_\e,\tilde u_k^-(\cdot,t_\e),D_x\varphi_\e)-\mathcal{I}^\kappa(y_\e,t_\e,\bar u_k^+(\cdot,t_\e),-D_y\varphi_\e),
\end{eqnarray}
and hence,  using \eqref{comp9uu+}, we see that we want develop appropriate estimates for the function
\begin{eqnarray}\label{eq:needest}
&&\varphi_\e(x_\e+ \eta_x^\e(z), y_\e +\eta_y^\e(z) ,t_\e) - \varphi_\e(x_\e, y_\e ,t_\e)\notag\\
 &&-\langle (\eta_x^\e(z), \eta_y^\e(z)),( D_x \varphi_\e(x_\e, y_\e ,t_\e), D_y \varphi_\e(x_\e, y_\e ,t_\e))\rangle,
\end{eqnarray}
where, as we recall,
\begin{equation*} \label{comp8}
\varphi_\e(x,y,t) = \frac{1}{2\e}|x-y|^{2} + \Lambda  (|x-\bar  {x}|^{2\gamma +2} +|y-\bar  {x}|^{2\gamma +2}) + \beta(t-\bar  {t})^2.
\end{equation*}
A straight forward calculation shows that \eqref{eq:needest} can be simplified to
\begin{eqnarray}\label{eq:bkappa1}
 &&\frac {1}{2 \e} |\eta_x^\e - \eta_y^\e|^2\notag\\
   &&+ \Lambda\left [ A_\e - (2+2\gamma)\left ( \lan \eta_x^\e, (x_\e-\bar x)\ran |x_\e-\bar x|^{2\gamma} + \lan \eta_y^\e, (y_\e-\bar x) \ran |y_\e-\bar x|^{2\gamma} \right )\right ],
\end{eqnarray}
where
$$A_\e:=\left [ |x_\e+\eta_x^\e- \bar x|^{2+2\gamma} - |x_\e- \bar x|^{2+2\gamma} +|y_\e+\eta_y^\e- \bar x|^{2+2\gamma} - |y_\e - \bar x|^{2+2\gamma} \right ].$$
We now first estimate the contribution to the expression in \eqref{eq:needestpp} coming from integration over $z\in B(0,1)\setminus B(0,\kappa)$, i.e., from the corresponding expression involving $\mathcal{I}_-^\kappa$. Given $z\in B(0,1)\setminus B(0,\kappa)$ and using \eqref{assump1elelnloc+}, we first see that
$$\frac{1}{2 \e} |\eta^\e_x - \eta^\e_y|^2 \leq \frac{c}{ \e}|z|^2 |x_\e-y_\e|^2.$$
Concerning the second term of \eqref{eq:bkappa1} we see, using the fundamental theorem of calculus and  \eqref{assump1elelnloc+}, that
\begin{eqnarray}\label{eq:bkappa1bb}
&&\left [ A_\e - (2+2\gamma)\left ( \lan \eta_x^\e, (x_\e-\bar x)\ran |x_\e-\bar x|^{2\gamma} + \lan \eta_y^\e, (y_\e-\bar x) \ran |y_\e-\bar x|^{2\gamma} \right )\right ]\notag\\
&&\leq c|z|^2(1+|x_\e-\bar x|^{2\gamma+2}+|y_\e-\bar x|^{2\gamma+2})
\end{eqnarray}
whenever $z\in B(0,1)\setminus B(0,\kappa)$. In particular, using the above estimates and \eqref{assump1elelnloc} we can conclude that
\begin{eqnarray}\label{eq:estkappa-}
&&| \mathcal{I}_-^\kappa(x_\e,t_\e,\tilde u_k^-(\cdot,t_\e),D_x\varphi_\e)-\mathcal{I}_-^\kappa(y_\e,t_\e,\bar u_k^+(\cdot,t_\e),-D_y\varphi_\e) | \notag \\
&&\leq     c \left (  \frac 1 \e |x_\e-y_\e|^2 + \Lambda (1+|x_\e-\bar x|^{2\gamma+2}+|y_\e-\bar x|^{2\gamma+2})\right ).
\end{eqnarray}
Finally, focusing on the contribution from the term $\mathcal{I}_+^\kappa$ we, through similar calculations as above now using \eqref{assump1elelnloc} on $\R^l \setminus B(0,1)$, we find that \eqref{eq:estkappa-} also holds with $\mathcal{I}_-^\kappa$ replaced by $\mathcal{I}_+^\kappa$ and hence the proof of \eqref{estnloc}, and Theorem \ref{Thm1-}, is complete. \hfill $\Box$

\setcounter{equation}{0} \setcounter{theorem}{0}
\section{Existence: proof of Theorem \ref{Thm1aa} and Theorem \ref{Thm1+}}\label{sec:exuni}
\subsection{Proof of Theorem \ref{Thm1aa}}To prove Theorem \ref{Thm1aa}  we use Perron's method. To construct a viscosity solution to problem \eqref{eq4+} using Perron's method we
define $u=(u_1, \dots, u_d)$ through
\begin{eqnarray*}
u_i\left(x,t\right):=\inf\{ u_i^+\left(x,t\right):\ \mbox{$(u_1^+, \dots, u_d^+)$ is a supersolution of problem \eqref{eq4+}}\}.
\end{eqnarray*}
Note that $u=(u_1, \dots, u_d)$ is well-defined based on the assumption of the existence of a barrier from above.
Let $u^\ast=(u_1^\ast, \dots, u_d^\ast)$ and $u_\ast=(u_{\ast,1}, \dots, u_{d,\ast})$ denote, respectively, the upper and lower semi-continuous envelopes of $u$, i.e.,
\begin{align*}
u^\ast_i\left(x,t\right)&=\limsup_{r\to 0}\{u_i\left(y,s\right):\ \left(y,s\right)\in \left(B\left(x,r\right)\times\left(t-r,t+r\right)\right)\cap (\mathbb R^N\times (0,T])\},\notag\\
u_{i,\ast}\left(x,t\right)&= - \left(- u_i\left(x,t\right)\right)^\ast,
\end{align*}
for $i\in\{1,...,d\}$. Then, by definition and by Theorem \ref{Thm1-} we have
\begin{eqnarray}\label{eq:storreochmindre}
u_i^- \leq  u_i^\ast, \quad u_{i,\ast} \leq u_i^+ \quad \textrm{and}\quad u_{i,\ast} \leq u_i^\ast,
\end{eqnarray}
for all $i\in\{1,...,d\}$,  and where $u^-=(u_1^-, \dots, u_d^-)$ and $u^+=(u_1^+, \dots, u_d^+)$ are sub- and supersolutions to \eqref{eq4+}, respectively.
The essence of Perron's method is now to prove that  $u^\ast$ and $u_\ast$ are, respectively,  a subsolution and a supersolution to problem \eqref{eq4+}. It then follows by the comparison principle that $u^\ast\leq u_\ast$ and hence $u = u_\ast = u^\ast$ is a viscosity solution to the system in \eqref{eq4+}. We first
prove that $u^\ast$ assumes the correct terminal data.  To do this, fix a component $i \in \{1,\dots,d\}$ and a point $y \in \R^N$. By assumption, there exists
a barrier from above  $u^+=u^{+,i,y}$ to the system in \eqref{eq4+} in the sense of Definition \ref{barr1}. In particular, by definition
$u_i(x,t) \leq u_i^{+,i,y,\e}(x,t)$ whenever $(x,t)\in\mathbb R^N\times [0,T]$ and for all $\e >0$. Furthermore, since $u^{+,i,y,\e}$ is continuous we have that $u_i^\ast(x,t) \leq (u_i^{+,i,y,\e})_i^\ast(x,t) = u_i^{+,i,y,\e}(x,t)$ for every $\e >0$ and $(x,t) \in \R^N \times [0,T]$.
In particular, for component $i$ and the point $y$ we have that
$$
u_i^\ast(y,T) \leq \lim _{\e \to 0} u_i^{+,i,y,\e}(y,T)= g_i(y).
$$
Since $i$ and $y$ are  arbitrary in this argument we conclude that $u^\ast$ assumes the correct terminal data. That $u_\ast$ also assumes the correct terminal data is proved similarly using the assumption of barriers from below. Hence it only remains to prove that the remaining conditions for the property of being sub- and supersolutions to the system hold for $u^\ast$ and $u_\ast$, respectively. However, after noticing that our switching costs are assumed to be continuous, see \eqref{assump2}, this follows in the same way, for the subsolution, as outlined on p. 70 in \cite{BJK10}, and for the supersolution, as outlined on p.71-72  in \cite{BJK10}. We omit further details and conclude that the proof of Theorem \ref{Thm1aa} is complete.

\subsection{Proof of Theorem \ref{Thm1+}}
In the light of Theorem \ref{Thm1aa} we see that to prove Theorem \ref{Thm1+} we only need to construct barriers from above and below, in the sense of Definition \ref{barr1}, for each $i \in \{1,\dots,d \}$ and $y\in \R^N$. To construct an appropriate barrier from above, for fixed $i \in \{1,\dots, d\}$ and $y \in \R^N$, we let, for all $j\in\{1,\dots,d\}$,
  \begin{eqnarray}\label{hupa-}
u_j^{+,i,y,\e}(x,t) &=& g(y)+\frac{ K}{\e^2}(T-t)\notag\\
&& + L(e^{\lambda(T-t)}+1)(|x-y|^2 + \e)^{\frac 12} + c_{i,j}(x,t),
\end{eqnarray}
where $K$ and $\lambda$ are non-negative degrees of freedom and $L$ is the Lipschitz-constant of $g(x)$.  We first note that \eqref{assump2el}, and the assumption that $g_i=g$ for all $i \in \{1,\dots, d\}$, implies that $c_{i,j}(x,T) \geq 0$ for all $i,j \in \{1,\dots, d\}$. Hence
$$u^{+,i,y,\e}(x,T) =   g(y)+ 2L(|x-y|^2 + \e)^{\frac 12} + c_{i,j}(x,T) \geq g(x),$$
whenever $x\in\mathbb R^N$ and, in particular,
$$u_i^{+,i,y,\e}(y,T) =   g(y)+ 2L \e^{\frac 12} + c_{i,i}(x,T).$$
Hence, using \eqref{assump3} $(i)$ we see that $u_i^{+,i,y,\e}(y,T) =   g(y)+2 L \e^{\frac 12}$ and $$\lim _{\e \to 0}u_i^{+,i,y,\e}(y,T) =   g(y).$$
We next verify that $u^{+,i,y,\e}$ is actually a supersolution to \eqref{eq4+}.
To do this we start by observing that assumption \eqref{assump3+} implies that $u^{+,i,y,\e}$ is above the obstacle.
In particular,
\begin{eqnarray*}
&&  u _j^{+,i,y,\e} - \max_{k \neq j}( -c_{j,k}(x,t) + u _k^{+,i,y,\e}) \notag \\
&=& - \max_{ k \neq j}( -c_{i,j}(x,t) -c_{j,k}(x,t)  + c_{i,k}(x,t)) \geq 0,
\end{eqnarray*}
where the last inequality is a consequence of \eqref{assump3+}.
Hence, to complete the proof we only have to verify that
\begin{eqnarray}\label{eq:NTSsupersolution}
 -\mathcal{H}u_j^{+,i,y,\e}(x,t)-\psi_j(x,t) &=&  -\mathcal{L}u_j^{+,i,y,\e}(x,t)\notag\\
 && - \mathcal{I}(x,t,u^{+,i,y,\e}_j) -\psi_j(x,t) \geq 0.
\end{eqnarray}
%Let $v(x,t):= (e^{\Lambda (T-t) }+1)(|x-y|^2 + \e)^{\frac 12}$.  Then,
%\begin{align}
%{\partial_{ x_k}v} & = (e^{\Lambda (T-t)} +1)\frac{(x_k - y_k)}{(|x-y|^2 + \e)^{\frac 1 2}} \notag \\
%{\partial_{ x_k  x_l}v}  =  {\partial_{ x_l  x_k}v}&= (e^{\Lambda (T-t)} +1) \frac{-(x_i-y_i)(x_j-y_j)}{(|x-y|^2+\e)^{\frac 3 2}} \notag \\
%{\partial_{ x_k  x_k}v} &= (e^{\Lambda (T-t)} +1) \frac{1}{(|x-y|^2+\e)^{\frac 1 2}} - (e^{\Lambda (T-t)} +1) \frac{(x_k-y_k)^2}{(|x-y|^2+\e)^{\frac 3 2}}.
%\end{align}
%Note first that $|\partial _{x_k}v| \leq  (e^{\Lambda (T-t)} +1) $ and $|{\partial_{ x_k  x_l}v}| \leq  (e^{\Lambda (T-t)} +1)  \frac {|x-y|^2}{(|x-y|^2 + \e)^\frac 32}$.
Using assumptions \eqref{assump1el}, \eqref{assump1elel}, and \eqref{assump2+} we deduce by standard calculations that
%\begin{align}
%&\left | \sum_{k,l=1}^N a_{k,l}(x,t)\partial_{x_k x_l} u^{+,i,y,\e} +\sum_{k=1}^N a_k(x,t)\partial_{x_k} u^{+,i,y,\e} \right | \notag \\
%% \leq &\left |\sum_{k,l=1}^N a_{k,l}(x,t)\partial_{x_k x_l} u^{+,i,y,\e} \right| + \left | \sum_{k=1}^N a_k(x,t)\partial_{x_k} u^{+,i,y,\e}\right |  \notag \\
%%\leq&  c(1+|x|^2)   \frac{ (e^{\Lambda (T-t)} +1) |x-y|^2}{(|x-y|^2+\e)^{\frac 3 2}} +   c(1+|x|^2)   \frac{(e^{\Lambda (T-t)} +1)}{(|x-y|^2+\e)^{\frac 1 2}} \notag \\
%&+ c(1+|x|) (e^{\Lambda (T-t)} +1). \notag
%\end{align}
%
%Now, for $|x-y|$ large, i.e. $|x-y|\geq 1$, the above yields
%\begin{align}
%&\left | \sum_{k,l=1}^N a_{k,l}(x,t)\partial_{x_k x_l} u^{+,i,y,\e} +\sum_{k=1}^N a_k(x,t)\partial_{x_k} u^{+,i,y,\e} \right | \notag \\
%\leq & c(1+|x|^2) \frac{(e^{\Lambda (T-t)} +1) }{|x-y|} + c(1+|x|) (e^{\Lambda (T-t)} +1) \notag \\
%\leq & c (e^{\Lambda (T-t)} +1) (1+|x|).
%\end{align}
%
%For $|x-y| < 1 $ we get in a similar way
%\begin{align}
%&\left | \sum_{k,l=1}^N a_{k,l}(x,t)\partial_{x_k x_l} u^{+,i,y,\e} +\sum_{k=1}^N a_k(x,t)\partial_{x_k} u^{+,i,y,\e} \right | \notag \\
% \leq&  c(y)  (e^{\Lambda (T-t)} +1)\frac{ 1 }{\e^{\frac 3 2}} +   c(y)  (e^{\Lambda (T-t)} +1)\frac{1}{\e^{\frac 1 2}} + c(y) (e^{\Lambda (T-t)} +1) \notag \\
%\leq & \frac {1}{\e^2} c(y)  (e^{\Lambda (T-t)} +1),
%\end{align}
%where $c(y)$ is a constant depending on the (fixed) point $y$.

\begin{align}
&\left | \sum_{k,l=1}^N a_{k,l}(x,t)\partial_{x_k x_l} u^{+,i,y,\e} +\sum_{k=1}^N a_k(x,t)\partial_{x_k} u^{+,i,y,\e} \right | \notag \\
\leq &  c  (e^{\Lambda (T-t)} +1) (1+ \frac {1}{\e^2} + |x|),
\end{align}
where the constant $c$ is independent from $K$.
%Therefore, we have that
%\begin{align}
%&\left | \sum_{k,l=1}^N a_{k,l}(x,t)\partial_{x_k x_l} u^{+,i,y,\e} +\sum_{k=1}^N a_k(x,t)\partial_{x_k} u^{+,i,y,\e} \right | \notag \\
% \leq & c (e^{\Lambda (T-t)} +1) (1+|x|) +\frac {1}{\e^2} c(y)  (e^{\Lambda (T-t)} +1) \notag \\
%\leq &c (e^{\Lambda (T-t)} +1) (\frac {1}{\e^2} +|x|) \leq 2c e^{\Lambda (T-t)} (\frac {1}{\e^2} +|x|).
%\end{align}
Concerning the non-local term $ \mathcal{I}(x,t,u^{+,i,y,\e}_j)$, using Taylor's formula and assumption \eqref{assump1elelnloc+}, we deduce, for $|z| \leq 1$,
\begin{eqnarray*}
&& \biggl |u^{+,i,y,\e}_j(x+\eta(x,t,z),t)-u^{+,i,y,\e}_j(x,t)-\sum_l\eta_l(x,t,z)\partial_{x_l}u^{+,i,y,\e}_j(x,t)\biggr |\notag\\
 &&\leq c|z|^2 \sup_{x' \in \R^N}||\nabla_{x}^2u^{+,i,y,\e}_j(x',t)||,
\end{eqnarray*}
 and, for $|z| > 1$,
\begin{eqnarray*}
&& \biggl |u^{+,i,y,\e}_j(x+\eta(x,t,z),t)-u^{+,i,y,\e}_j(x,t)-\sum_l\eta_l(x,t,z)\partial_{x_l}u^{+,i,y,\e}_j(x,t)\biggr |\notag\\
 &&\leq c \sup_{x' \in \R^N}||\nabla_{x}^2u^{+,i,y,\e}_j(x',t)||.
\end{eqnarray*}
%
%Recall that
%\begin{align}
%{\partial_{ x_k  x_l}v}  =  {\partial_{ x_l  x_k}v}&= (e^{\Lambda (T-t)} +1) \frac{-(x_i-y_i)(x_j-y_j)}{(|x-y|^2+\e)^{\frac 3 2}} \notag \\
%{\partial_{ x_k  x_k}v} &= (e^{\Lambda (T-t)} +1) \frac{1}{(|x-y|^2+\e)^{\frac 1 2}} - (e^{\Lambda (T-t)} +1) \frac{(x_k-y_k)^2}{(|x-y|^2+\e)^{\frac 3 2}}.
%\end{align}
It is clear from the construction of $u^{+,i,y,\e}_j$, see \eqref{hupa-}, and \eqref{assump2+} that
$$ \sup_{x' \in \R^N \setminus B(y,1)} ||\nabla_{x}^2u^{+,i,y,\e}_j(x',t)|| \leq c$$
and hence we only need to consider values of $x$ such that $|x - y| < 1$. Straightforward calculations using \eqref{assump1elelnloc+}  show that
\begin{align*}
\sup_{x' \in B(y,1)}||\nabla_{x}^2u^{+,i,y,\e}_j(x',t)|| \leq c (e^{\Lambda (T-t)} +1) \left (\frac{1}{\e^{\frac{3}{2}}} +\frac{1}{\e^{\frac{1}{2}}} \right ) \leq \frac{c}{\e^{2}} (e^{\Lambda (T-t)} +1),
\end{align*}
where, again, the constant $c$ is independent from $K$. Hence, taking assumption \eqref{assump1elelnloc} into account we see that
\begin{align*}
|  - \mathcal{I}(x,t,u^{+,i,y,\e}_j)| \leq \frac{c}{\e^{2}} (e^{\Lambda (T-t)} +1).
\end{align*}
Next, note that
\begin{align*}
\partial_t u_j^{+,i,y,\e} & = -\frac{K}{\e^2} - \Lambda e^{\Lambda(T-t)}(|x-y|^2+\ep)^{\frac 1 2} + \partial_t c_{i,j}.
\end{align*}
Putting the above estimates together and using \eqref{assump2+} we find that
\begin{align}
\mathcal{H}u_j^{+,i,y,\e}(x,t)&-\psi_j(x,t)  \notag \\
\geq& \frac{K}{\e^2} + \Lambda e^{\Lambda(T-t)}(|x-y|^2+\ep)^{\frac 1 2} - c e^{\Lambda (T-t)} (1+ \frac {1}{\e^2} +|x|)  \geq 0
\end{align}
for $K$ and $\Lambda$ large enough. Hence $u^{+,i,y,\e}$ is a supersolution to \eqref{eq4+}.

To construct an appropriate barrier from below, for fixed $i \in \{1,\dots, d\}$ and $y \in \R^N$, we let, for all $j\in\{1,\dots,d\}$,
  \begin{eqnarray}\label{hupa+}
u_j^{-,i,y,\e}(x,t) &=& g(y)- \frac{K}{\e^2}(T-t)\notag\\
&& - L(e^{\lambda(T-t)}+1)(|x-y|^2 + \e)^{\frac 12} - c_{i,j}(x,t),
\end{eqnarray}
where again $K$ and $\lambda$ are non-negative degrees of freedom and $L$ is the Lipschitz-constant of $g(x)$.  The above argument can then be repeated to conclude that $u^{-,i,y,\e}$ is a subsolution to \eqref{eq4+} and hence the proof of Theorem \ref{Thm1+} is complete.  We omit further details.\hfill $\Box$

\begin{remark} Note that a local version of Theorem \ref{Thm1+} is proved as Theorem 2.4 in \cite{LNO12}. In the proof of Theorem 2.4 in \cite{LNO12}
the barriers in \eqref{hupa-} and \eqref{hupa+} are also used. However, there are two typo errors in the statements of these barriers in
\cite{LNO12}. Indeed, the factors $e^{-\lambda t}$ and $K$ stated in the corresponding construction in \cite{LNO12} should be corrected and replaced by $e^{\lambda(T-t)}$ and $\frac{K}{\e^2}$ respectively, as above. The subsequent calculation/argument in \cite{LNO12} should also be modestly adjusted accordingly.
\end{remark}

\end{document}